\documentclass{amsart}
\usepackage{amssymb}
\usepackage{graphicx}
\usepackage[usenames]{color}
\newtheorem{ass}{Assumption}

\def \R{I\!\!R}

\input{tcilatex}

\begin{document}
\title[population growth]{On
population growth with catastrophes}
\author{Branda Goncalves \and Thierry  Huillet \and Eva L\"ocherbach}
\address{B. Goncalves and T. Huillet: Laboratoire de Physique Th\'{e}orique et Mod\'{e}lisation, CY Cergy Paris Universit\'{e}, CNRS UMR-8089, 2 avenue Adolphe-Chauvin, 95302 Cergy-Pontoise, FRANCE\\
E-mails: branda.goncalves@outlook.fr, Thierry.Huillet@cyu.fr}
\address{E. L\"ocherbach: SAMM, Statistique, Analyse et Mod\'elisation Multidisciplinaire, Universit\'e Paris 1 Panth\'eon-Sorbonne, EA 4543 et FR FP2M 2036 CNRS, France.\\
E-mail: eva.locherbach@univ-paris1.fr}
\maketitle

\begin{abstract}
In this paper we study a  particular class of Piecewise deterministic Markov processes (PDMP's) which are semi-stochastic catastrophe versions of deterministic population growth models. In between successive jumps the process follows a flow describing deterministic population growth. Moreover, at random jump times, governed by state-dependent rates, the size of the population shrinks by a random amount of its current size, an event possibly leading to instantaneous local (or total) extinction. A special  separable shrinkage transition kernel is investigated in more detail, including the case of total disasters. We discuss conditions under which such processes are recurrent (positive or null) or transient. To do so, we introduce a modified scale function which is  used to compute, when relevant, the law of the height of excursions and to decide if the process is recurrent or not. The question of the finiteness of the time to extinction is investigated together with the evaluation of the mean time to extinction when the last one is finite. Some information on the embedded jump chain of the PDMP is also required when dealing with the classification of states $0$ and $\infty $ that we exhibit.
\end{abstract}

\section{Introduction} 
In this paper we study population growth models subject to random catastrophes, designed to describe for instance the evolution of a disease, the growth of a market or of the capital of a company, .... 
In our model, catastrophes appear after unpredictable  random times. These random times are generalized exponentially distributed times having rate  
 $ \beta ( x) $ whenever the current size of the process is $x.$ When a catastrophe happens, the process shrinks by a random amount of its size,
an event possibly leading to instantaneous local extinction (extinction of the disease, collapse of the market, failure of the company, ...). The successive catastrophe events are the only jumps  of the system. In between these jumps,  the process follows a flow describing 
deterministic growth, given by
$$ x_t = x_0 + \int_0^t \alpha ( x_s ) ds ,$$
with a locally Lipschitz continuous drift function $ \alpha $ which is strictly positive on $ (0, \infty ). $  
This leads to a resulting strong Markov process $ X = (X_t)_{t \geq 0}. $  All features of its dynamic are gathered in its infinitesimal generator given for smooth test functions $u$ by 
\begin{equation}\label{eq:gener0}
 G u ( x) = \alpha ( x) u' ( x) + \beta(x) \int_0^x H (x, dy ) [ u(y) - u(x) ], \; x \geq 0, 
\end{equation} 
where $H( x, dy ) $ is the jump kernel giving the after-jump position $y \in [0, x ], $ provided the current size of the population before undergoing the catastrophe is $x.$

Such processes are a particular instance of piecewise deterministic Markov processes (PDMP's) as introduced in \cite{Davis}; 
models of a similar flavor were considered in \cite{BGR}, 
\cite{BGR2}, \cite{BRT}, \cite{HT}, \cite{HT2}, \cite{PTB} and \cite{T}, see
also \cite{Box}, \cite{GG}, \cite{GG2} and \cite{neuts}.

The process $(X_t)_t $ takes values in $[0, \infty ],$ and both boundaries $0$ and $ + \infty $ have to be added to the state space. Indeed, since $\alpha$ is only locally Lipschitz continuous,  the deterministic flow may reach $ +\infty $ in finite time. Moreover, $ \alpha ( 0 ) $ may equal $0,$ such that it is possible that the process gets stuck in $0.$ We therefore propose in a first step the classification of the two boundaries together with necessary and sufficient conditions ensuring that the process is of finite activity, that is non-explosive in the sense that the number of jumps per finite time interval is finite almost surely. 

The main part of our paper is devoted to the study of the return times to $0$ of the process. The question whether the process almost surely returns to $0$ and how long it takes to do so is of course of tremendous importance in any application. On the other hand, under the assumption that $0$ is reflecting, the return times to $0$ induce a basic regeneration scheme and therefore trivially imply the recurrence of the process. 

Recurrence of one-dimensional Markov processes which are regular such as diffusions is usually studied by means of the associated scale function and speed measure. For example, one-dimensional elliptic diffusions are known the be recurrent if and only if their associated scale function (that is, a function transforming the process into a local martingale) is a space-transform, that is, a bijection, see e.g. \cite{Has} Example 2 in Section 3.8. Once the scale function is explicitly known, exit probabilities of bounded intervals can be computed. Hitting time moments are also known, expressed in terms of the Green's function and on the speed measure. 

Contrarily to the case of one-dimensional elliptic diffusions, the class of PDMP's we study in this paper is very irregular. A very particular feature of our process is the following intrinsic asymmetry. The only way to go up is by deterministic continuous motion -- therefore exit times of intervals $ [0, b ] $ will always be hitting times of level $b $ - while the process does only go down by jumps -- exit times of intervals $ [a, \infty ] $ will always be jumps. Another difficulty comes from the fact that in general the process $(X_t)_t $ cannot possess other bijective scale functions $s(\cdot ) $ than the constant ones. Indeed, starting from $0$ and under the assumption that $0$ is reflecting, $ X_t > 0 $ almost surely for any $ t > 0, $ such that $ s (X_t) > s (0) $ as well -- which obviously contradicts the martingale property. We are however able to define a {\it modified scale function} of the process that does not transform the trajectory into a martingale but allows to completely characterize the recurrence of the process and to compute certain exit probabilities. This function exists in the situation when the jump kernel is separable, that is, $H(x, y ) := \int_{[0, y ]} H (x, dz) $ is of the form $ H(x, y ) = h(y)/h(x) $ for some positive, non-decreasing function,  for any $ 0 \le y \le x . $ 

In this separable case, our first main result, Theorem \ref{prop:scale}, gives an explicit formula for $p(x, b ) , $ the probability of hitting $0$ before hitting $b, $ starting from $ x \in ( 0, b),$ under suitable conditions on the coefficients of the system. Theorem \ref{prop:scale} also gives the explicit value of $ p( 0, b ) = \lim_{ x \to 0} p ( x, b ) . $ Unlike in the diffusion case, $ p(0, b ) $ does not equal $1$ but is related to the distribution function of the height of an excursion, that is, the maximal value of the process in between two successive visits to $0.$ More precisely, 
$$ p (x, b) = \frac{ s(b) - s(x) }{ s(b ) + \frac{1}{h(0)}} , $$
where $ s$ is the  {\it modified scale function} of the process, given by 
\begin{equation}\label{eq:sbis}
s(x) = \int_0^x \frac{\gamma \left( y\right) }{h\left(
y\right) }e^{\Gamma \left( y\right) }dy , \mbox{ with } \Gamma( x) = \int_0^x \frac{ \beta }{\alpha} ( y) dy.
\end{equation}

Although $s$ is not a true scale function of $X, $ the recurrence of $X$ is however equivalent to the fact that $s$ is a bijection, that is, $s( \infty ) = \infty .$ Therefore we recover the same characterization of recurrence as in the case of one-dimensional elliptic diffusions, at least if $0$ is reflecting and accessible and $ +\infty $ inaccessible.  This result is stated in Proposition \ref{cor:14}.

Our second main result, Theorem \ref{prop:u(0)}, gives then the expected length of an excursion out of $0 - $ that is, of the expected time it takes the process to come back to $0, $ starting from there. To obtain this result, we rely on the fundamental formula relating the invariant measure of a process to the expected occupation time of a given set in between successive visits to a recurrent state (here, $0$). This allows to recover the length of an excursion by means of the speed density $ \pi $ and the {\it expected local time in $ 0$} of the process during one excursion. 


{\bf Organization of the paper.} In Section 2 we introduce our model and discuss some first properties, including the distribution of the first jump time, the classification of the boundary states $ 0 $ and $ \infty $ and a discussion of the non-explosion of the stochastic process in Proposition \ref{prop:explos} and Proposition \ref{prop:explos2}. Section 3 is devoted to the study of some basic regularity properties of the associated transition semigroup. In particular, we show that the ``noise" which is present in the random choices of the jump times regularizes in the sense that $ {\mathcal L} (X_t| X_0 = x ) $ is absolutely continuous with respect to the Lebesgue measure on $ [0, x_t (x) ), $ see Proposition \ref{prop:2}. We also provide an explicit formula for the speed measure and its density in \eqref{eq:speed}. Section 4 contains the main results of the paper related to the recurrence and the return times to $0.$ Finally, in Section 5 we present some simulation results.

\section{Model definition and first results}
We consider a piecewise deterministic Markov process $ X_t $ taking values in $ [0, \infty ], $ describing the random size of a population. The dynamic of the process is given by two main ingredients. 
Firstly, in between successive jumps, the size of the population grows in a deterministic way, described by a deterministic flow. Secondly, at some random jump  times, catastrophe events occur at which the current size of the 
population shrinks by a random amount. 

We start by discussing the deterministic growth part in between the successive jumps. 
\subsection{Deterministic population growth models}
The evolution of the population size in between successive jumps follows the dynamic $\overset{.}{x}_{t}=\alpha \left( x_{t}\right) $, $x_{0}=x \geq 0 .$
Throughout this paper, the drift function $ \alpha $ is supposed to be continuous on $\left[ 0,\infty \right) $ and 
positive on $\left( 0,\infty \right) . $ 
\subsubsection{Algebraic growth models}
With $\alpha _{1},$ $a>0$, consider the
growth dynamics 
\begin{equation}\label{PG1}
\overset{.}{x}_{t}=\alpha _{1}x_{t}^{a}\text{, }x_{0}=x,  
\end{equation}
for some growth field $\alpha \left( x\right) =\alpha _{1}x^{a}$.
Note that in this case $\alpha \left( x\right) $ is increasing with $x.$ 
Integrating when $a\neq 1$ (the non linear case), we get formally 
\begin{equation}
x_{t}\left( x\right) =\left( x^{1-a}+\alpha _{1}\left( 1-a\right) t\right)
^{1/\left( 1-a\right) }.  \label{PG1I}
\end{equation}
In principle, such growth models are considered for some positive initial
condition $x> 0 .$ Because we will deal in the sequel with catastrophic events
that can send the population to state $0$, it is also important to consider
such growth models when started at $x=0$. Either after
hitting state $0$, the population remains stuck to $0,$ and in this case $0$ is absorbing. 
Or the population can regenerate starting afresh from $0,$ and $0$ is reflecting.

Three cases arise:

\begin{itemize}
\item $0<a<1$: then $x\geq 0$ makes sense and in view of $1/\left(
1-a\right) >1$, the growth of $x_{t}$ is algebraic at rate larger than $1$. When $x=0$%
, the dynamics has two solutions, one $x_{t}\left( 0\right) \equiv 0$ for $%
t\geq 0$ and the other $x_{t}\left( 0\right) =\left( \alpha _{1}\left(
1-a\right) t\right) ^{1/\left( 1-a\right) }$ because the velocity field $%
\alpha \left( x\right) $ in (\ref{PG1}) with $\alpha \left( 0\right) =0$ is
not Lipschitz as $x$ gets close to $0$, having an infinite derivative. The
solution $x_{t}\left( 0\right) =\left( \alpha _{1}\left( 1-a\right) t\right)
^{1/\left( 1-a\right) }$ with $x=0$ reflects some spontaneous generation
phenomenon: following this path, the mass at time $t>0$ is not $0$, although
initially it was. Whenever the spontaneous generation phenomenon
holds, we shall say that state $0$ is reflecting. In what follows, without explicitly mentioning it, we shall always choose this second, maximal solution describing spontaneous generation of mass. 
\item  $a>1$: then $x>0$ only makes sense and  $%
x\left( t\right) $ reaches state $+\infty $ in finite time $I_\infty (x) =x^{1-a}/\left[
\alpha _{1}\left( a-1\right) \right] $. We get
\begin{equation*}
x_{t}\left( x\right) =x\left( 1-t/I_\infty \left( x\right) \right) ^{1/\left(
1-a\right) },
\end{equation*}
with algebraic singularity. Whenever a growth
process reaches state $ + \infty $ in finite time, we shall say that state $\infty $ is accessible. 
\item $a=1$: this is a simple special case not treated in (\ref{PG1I}),
strictly speaking. However, expanding the solution (\ref{PG1I}) in the
leading powers of $1-a$\ yields consistently: 
\begin{equation}
\begin{array}{c}
x_{t}\left( x\right) =e^{\log \left( x^{1-a}+\alpha _{1}(1-a)t\right)
/\left( 1-a\right) } \\ 
=e^{\log [x^{1-a}\left( 1+\alpha _{1}x^{a-1}(1-a)t\right) ]/\left(
1-a\right) }\sim xe^{(1/(1-a))\alpha _{1}x^{a-1}(1-a)t}\sim xe^{\alpha
_{1}t}.
\end{array}
\label{A1}
\end{equation}
Here $x\geq 0$ makes sense for (\ref{PG1}) with $x_{t}\left( x\right)
=xe^{\alpha _{1}t}$ for $t\geq 0$ if $x\geq 0$. This is the simple Malthus
growth model. 
\end{itemize}

\subsubsection{The role of $0$ and of $ +\infty $}
In general, $\alpha $ being positive on $\left( 0,\infty \right) ,$ we have
\begin{equation*}
\int_{x}^{x_{t}\left( x\right) }\frac{dy}{\alpha \left( y\right) }=t.
\end{equation*}
Notice that in particular $t^{\prime }>t\geq 0$ entails $x_{t^{\prime }}\left( x\right)
>x_{t}\left( x\right) $, provided $ x > 0 $ and $ x_{t' (x) }< \infty . $ 

If for $x>0,$ $I_{0}\left( x\right) :=\int_{0}^{x}\frac{dy}{\alpha \left(
y\right) }<\infty ,$ then we have
\begin{equation*}
x_{t}\left( x\right) =I_{0}^{-1}\left( I_{0}\left( x\right) +t\right).
\end{equation*}

If for $x>0,$ $I_{0}\left( x\right) =\infty $ and $I_{\infty }\left(
x\right) :=\int_{x}^{\infty }\frac{dy}{\alpha \left( y\right) }<\infty ,$ then 
\begin{equation*}
x_{t}\left( x\right) =I_{\infty }^{-1}\left( I_{\infty }\left( x\right)
-t\right).
\end{equation*}
Finally we have in all cases,
\begin{equation*}
x_{t}\left( x\right) =I^{-1}\left( I\left( x\right) +t\right) ,
\end{equation*}
where $I\left( x\right) =\int^{x}\frac{dy}{\alpha \left( y\right) }$ is an
indeterminate integral. This occurs for example when $\alpha \left( x\right)
=x^{a}e^{-bx}$ with $a>1$ and $b>0.$

Clearly, $I_{0}\left( x\right) $ is the time needed to
reach some state $x$ inside the domain $\left( 0,\infty \right) $ starting from $0,$ and  $I_{\infty }\left( x\right) $
 the time needed to reach $\infty $ starting from
some $x$  inside the domain. Thus 
\begin{eqnarray*}
I_{0}\left( x\right) &<&\infty \Longleftrightarrow \text{ state }0\text{ is
reflecting, }I_{\infty }\left( x\right) <\infty \Longleftrightarrow \text{
state }\infty \text{ is accessible,} \\
I_{0}\left( x\right) &=&\infty \Longleftrightarrow \text{ state }0\text{ is
absorbing, }I_{\infty }\left( x\right) =\infty \Longleftrightarrow \text{
state }\infty \text{ is inaccessible.}
\end{eqnarray*}

\subsection{Adding catastrophes}
We now consider the stochastic process $X_t$ that follows the deterministic flow with drift $ \alpha $ and jumps at position dependent rate $\beta  $ which is a continuous function on $ [ 0,\infty ), $ 
positive on $ ( 0,\infty ) .$ At the jump times, the size of the
population shrinks by a random amount $\Delta \left( X_{t-}\right) \in ( 0, X_{t-} ] $ of its
current size $X_{t-}$. Up to the next jump time, $X$ grows following the
deterministic dynamics started at $Y (X_{t-}) :=X_{t-}-\Delta
\left( X_{t-}\right) $. 

Let
\begin{equation*}
\mathbf{P}\left( X\leq y\mid X_{-}=x\right) =\mathbf{P}\left(  \Delta (x) \geq x-y\right)=  H\left( x,y\right)  , 0 \le y \le x , 
\end{equation*}
be the kernel $H$ which fixes the law of the jump amplitude. $H\left( x,y\right) $ is a
non-decreasing function of $y$ with $H\left( x,y\right) =1$ for all $y\geq
x. $ We shall also write
$$
H\left( x,dy\right)  =H\left( x,0\right) \delta _{0}+\overline{H}\left(
x,dy\right) ,  \; 
H\left( x,y\right)  =\int_{0}^{y}H\left( x,dy^{\prime }\right) =H\left(
x,0\right) +\overline{H}\left( x,y\right) ,
$$
with $\overline{H}\left( x,0\right) =0$, $\overline{H}\left( x,x\right)
=1-H\left( x,0\right) .$ If $H\left( x,0\right) >0, $ there
is a positive probability of disasters (instantaneous local extinction).

A special (separable) interesting case is when 
\begin{equation*}
H\left( x,y\right) \overset{*}{=}\frac{h\left( y\right) }{h\left( x\right) }=%
\frac{h\left( 0\right) }{h\left( x\right) }+\frac{h\left( y\right) -h\left(
0\right) }{h\left( x\right) }, 
\end{equation*}
for some positive non-decreasing right-continuous function $h.$

Our main concern will deal with this particular separable structure of $H.$ In this case,
necessarily $x\rightarrow H(x,y)$\ is non-increasing in $x$\ for all $y$\
(because $y\rightarrow H(x,y)$\ is non-decreasing in $y$\ for all $x$
entailing $h$\ non-decreasing).

\begin{example}
\begin{itemize}
\item
Examples for the separable case are:
\begin{itemize}
\item $h(x)=e^{x}$ in which case $H(x,0)=e^{-x}>0$ (instantaneous disaster can
occur with some positive probability)$.$ This it the continuous version of 
the truncated geometric model defined in \cite{neuts}.

Letting $Z>0$\ random, with cpdf $\overline{F}_{Z}\left( z\right) ={\bf P}%
\left( Z>z\right) $, $H\left( x,y\right) =\overline{F}_{Z}\left( x\right) /%
\overline{F}_{Z}\left( y\right) $ is also in this class, with $H\left(
x,0\right) =\overline{F}_{Z}\left( x\right) >0.$

\item $h(x)=x$ in which case $H(x,0)=0$ (no instantaneous disaster)$.$

In the latter two examples $H(\infty ,y)=0$ and there is no way to come down
from infinity.

\item Let $Z>0$\ random and proper, with pdf $\overline{F}_{Z}\left( z\right) =%
{\bf P}\left( Z>z\right) $. Suppose $H\left( x,y\right) =h(y)/h(x)$ with $%
h(x)=h\left( \infty \right) -\left( h\left( \infty \right) -h\left( 0\right)
\right) \overline{F}_{Z}\left( x\right) $, for some constants $\infty
>h\left( \infty \right) >h\left( 0\right) >0$. Then, $h(x)$ being bounded
above, $H(\infty ,y)=1-\frac{h\left( \infty \right) -h\left( 0\right) }{%
h\left( \infty \right) }\overline{F}_{Z}\left( y\right) $ and there is a
possibility to come down from infinity. Note $H\left( x,0\right)
=h(0)/h(x)>0.$
\end{itemize}
\item
Examples for non separable kernels are:
\begin{itemize}
\item $H\left(
x,dy\right) =\delta _{ux} (dy) ,$ for some $u\in \left( 0,1\right) .$ After each catastrophe a fixed fraction $u$\ of
the previous population is kept. 

\item Let $U\in \left( 0,1\right) $\ random, with pdf $F_{U}\left( u\right) =%
{\bf P}\left( U\leq u\right) .$\ Define{\bf \ }$H\left( x,y\right)
=F_{U}\left( \frac{y}{x}\right) .${\bf \ }After each catastrophe a random
fraction $U$\ of the previous population is kept.
\end{itemize}
\end{itemize}
\end{example}

\subsection{Representing the process as solution of a stochastic differential equation driven by a Poisson random measure}
Introducing a Poisson random measure $M\left( dt,dz \right) $ on $\left[ 0,\infty
\right) \times \left[ 0,\infty \right)  $ with intensity $ dt dz ,$ we are thus led to consider the piecewise deterministic Markov process
(PDMP) $(X_t)_{t \geq 0} $ with state-space $\left[ 0,\infty \right] $ obeying
\begin{equation}
dX_{t}=\alpha \left( X_{t-} \right) dt - \Delta ( X_{t-}  )  
\int_{0}^{\infty } \mathbf{1}_{\left\{ z \leq \beta \left( X_{t-}  \right) \right\} }M\left( dt,dz  \right) ,
\end{equation}
$ X_0  = x \geq 0 .$ The associated infinitesimal generator is given for any smooth test function $u$ by 
\begin{equation}\label{eq:G}
 G u (x) = \alpha (x) u' ( x) + \beta ( x) \int_0^x [ u(y) - u(x) ]  H ( x, dy ) , x \geq 0.
\end{equation} 
In the separable case $ H(x,y ) = h(y)/h(x) $ for all $ 0 \le y \le x, $ this reads 
\begin{equation}\label{eq:Gsep}
G u (x) = \alpha (x) u' ( x) - \beta ( x)/h(x)  \int_0^x u'(y)h(y) dy   , x \geq 0.
\end{equation} 

Notice that $ t \to X_t $ is non-decreasing in between successive jumps such that the only possibility for the process to go down is by jumping. The underlying jump counting process  is 
\begin{equation}\label{eq:nt}
dN_{t}=\int_{0}^{\infty } 1_{\left\{ z\leq \beta \left( X_{t-}  \right) \right\} }M\left( dt,dz  \right) , \; \mbox{ with } \mathbf{E} \left( N_{t} \right) =\mathbf{E}\int_{0}^{t}\beta
\left( X_{s} \right) ds. 
\end{equation}

As usual, to emphasize the dependence on the starting position, we shall write $ \mathbf{P}_x$ and $ \mathbf{E}_x $ for the probability and its associated expectation on the event when $ X_0 = x.$ 
Defining 
\begin{equation}\label{eq:deftx}
T_{x}=\inf \{ t >  0 : X_t \neq X_{t-}  | X_0 = x  \}
\end{equation}
(with the convention that $ \inf \emptyset = \infty  $), $T_x$ is the time at which a first jump occurs, when the process starts from $x.$ In what follows, we shall write $ S_0 = 0 \le  S_1 \le  S_2 \le  \ldots \le  S_n $ for the successive jump times of the process $ X_t.$ Notice that $ S_1 = T_x, $ if $ X_0= x.$ Moreover, conditionally on $ X_{S_1} = x_1, $   $ S_2 - S_1 \stackrel{\mathcal L}{=} T_{ X_{x_1} } , $ etc. 

We shall also consider  
\begin{equation*}
\tau _{x,0}=\inf \left\{ t \,  >  \, 0:X_{t} =0 | X_0 = x \right\} , \, \inf \emptyset := +\infty , 
\end{equation*}
which is the first time to local extinction. We are led to the following distinctions:

1/ Total catastrophes (disasters): 
$$H\left( y,0\right) =1 \mbox{ for all $ y > 0,$} $$
which means that  $\mathbf{P}%
\left( X_{T_x }=0\mid X_{T_x -}=y\right) = \mathbf{P} ( \Delta ( y)  = y)  =1.$  

Given $x>0$, state $0$ is reached with probability $1,$ provided $ T_x < \infty $ almost surely.

- If $0$ is absorbing for $x_{t}$, then $X_{t}=0$ for all $t\geq T_{x}.$
Moreover $T_{x}$ coincides with the first time to extinction $\tau _{x,0}.$

- If $0$ is reflecting for $x_{t}$, $X_{t}$ possibly visits $0$ a finite or
an infinite number of times depending on weather $ T_x < \infty $ almost surely or not.

2/ Partial catastrophes (catastrophes without disasters): 
$$H\left(x,0\right) =0 \mbox{ for all }  x > 0 ,$$ 
which is equivalent to  $\mathbf{P} ( \Delta ( x) < x ) =1$ for all $x > 0 .$

Given $x>0$, state $0$ is never visited. The reflecting/absorbing status of
state $0$ is unimportant, being never reached. Formally, $\tau _{x,0}=\infty 
$. 

3/ General catastrophes: 
$$H\left( x,0\right) \in \left( 0,1\right) ,$$ 
which means that  $%
\mathbf{P} ( \Delta ( x) < x ) \in (0, 1 ) $ for all $ x > 0.$ Then $%
\mathbf{P}\left( X>0\mid X_{-}=x\right) =\mathbf{P} ( x - \Delta (x) > 0 ) = 1-H\left( x,0\right) \in \left(
0,1\right) .$

- If $0$ is absorbing for $x_{t}$, $X_{t}=0$ for all $t\geq \tau _{x,0},$ where $%
\tau _{x,0}$ is stochastically larger than $T_{x}.$

- If $0$ is reflecting for $x_{t}$, $X_{t}$ possibly visits $0$ a finite or
an infinite number of times.

\subsection{First jump distribution in case of $I_\infty (x) = \infty $}
In this subsection we suppose that $I_\infty (x) = \infty  $ such that the deterministic flow does not reach state $ + \infty $ in finite time. Supposing $ X_0 = x, $ since $ X_t = x_t (x) $ on $ t < T_x,$ we have
\begin{equation*}
\mathbf{P}\left( T_{x}>t\right) =\mathbf{P}_{x}\left( N_{t}=0\right) =%
\mathbf{P}_{x}\left( \int_{0}^{t}\int_{0}^{\infty } \mathbf{1}_{\left\{ z
\leq \beta \left( x_{s} (x) \right) \right\} }M\left( ds,d z \right) =0\right),
\end{equation*}
where $N_t $ was defined in (\ref{eq:nt}) above.
With $\gamma \left( x\right) :=\beta \left( x\right) /\alpha \left( x\right) 
$ and $\Gamma \left( x\right) :=\int^{x}\gamma \left( y\right) dy$, an
increasing function defined as an indefinite integral,  we get, since $ \alpha > 0 $ on 
$ (0, \infty ),  $
\begin{equation}\label{eq:tx0}
\mathbf{P}\left( T_{x}>t\right) =e^{-\int_{0}^{t}\beta \left( x_{s}\left(
x\right) \right) ds}=e^{-\left[ \Gamma \left( x_{t}\left( x\right) \right)
-\Gamma \left( x\right) \right] } , \mbox{ for all } t \geq 0. 
\end{equation}
This leads naturally to the introduction of the following two conditions. 
\begin{ass}\label{ass:gammainfty}
$\Gamma \left( \infty
\right) =\infty .$
\end{ass}

\begin{ass}\label{ass:gamma0}
$\Gamma \left( 0\right) >-\infty . $
\end{ass}

Notice that imposing Assumption \ref{ass:gammainfty} ensures $\mathbf{P}\left( T_{x}<\infty \right)
=1 .$ Indeed, since $ \alpha > 0 $ on $ (0, \infty ) , $ for any $ x > 0,$ $x_t ( x)  \to \infty $ as $ t \to \infty ,$ which, together with \eqref{eq:tx0} allows to conclude.  

Moreover, imposing Assumption \ref{ass:gamma0}  implies that for all $ t\geq 0,$ $\lim_{x\rightarrow 0}\mathbf{P}\left( T_{x}>t\right) >0$ (this is condition
2.6 in \cite{GG}). If $0 $ is reflecting, the definition of $ T_0$ in \eqref{eq:deftx} makes sense replacing $x$ by $0,$ and \eqref{eq:tx0} remains valid, since $ t \mapsto x_t (0) $ is invertible. In this case, Assumption \ref{ass:gamma0} is automatically satisfied. 

Under Assumption \ref{ass:gammainfty}, since $I_\infty ( x) = \infty $ by assumption, we obtain for $x > 0 $ 
\begin{equation}\label{eq:txesp}
\mathbf{E}\left( T_{x}\right) =\int_{0}^{\infty }e^{-\int_{x}^{x_{t}\left(
x\right) }\gamma \left( y\right) dy} dt =\int_{x}^{\infty }\frac{1}{\alpha
\left( z\right) }e^{-\int_{x}^{z}\gamma \left( y\right) dy} dz 
=e^{\Gamma \left(
x\right) }\int_{x}^{\infty }\frac{1}{\alpha \left( z\right) }e^{-\Gamma
\left( z\right) } dz  .
\end{equation}
Notice that the above expression is finite if we assume that  $ \beta $ is lower-bounded in a neighborhood of $ \infty ,$ say by a strictly positive constant $ c>0 .$ Then for $ x $ sufficiently large, 
$$
 \mathbf{E}\left( T_{x}\right)= e^{\Gamma \left(
x\right) }\int_{x}^{\infty }\frac{1}{\alpha \left( z\right) }e^{-\Gamma
\left( z\right) } dz=e^{\Gamma \left( x\right) }\int_{x}^{\infty }\frac{dz}{\beta \left( z\right) }\gamma \left( z\right) e^{-\Gamma \left(
z\right) } 
\leq \frac{1}{c}\int_x^{\infty}\gamma(z)e^{- \Gamma(z)}dz < \infty .
$$

\begin{example}
We take  $\alpha \left( x\right) =\alpha _{1}x^{a}$ with $a\leq 1$
 such that state $\infty $ is inaccessible. Moreover we choose $\beta \left( x\right) =\beta
_{1}x^{b}$ with $b>a-1,$ implying that $\gamma \left( x\right) =\gamma
_{1}x^{b-a}$\emph{, }$\Gamma \left( x\right) =\int_0^{x}\gamma \left( y\right)
dy=\frac{\gamma _{1}}{b-a+1}x^{b-a+1}, $ where we have chosen $\Gamma \left( 0\right) =0.$ Notice that 
$\Gamma \left( \infty \right) =\infty $ and 
\begin{multline*}
\Gamma \left(
x_{t}\left( x\right) \right) -\Gamma \left( x\right) =\frac{\gamma _{1}}{%
b-a+1}\left[ y^{b-a+1}\right] _{x}^{x_{t}\left( x\right) }=\frac{\gamma _{1}%
}{b-a+1}\left( x_{t}\left( x\right) ^{b-a+1}-x^{b-a+1}\right) \\
=\frac{\gamma
_{1}}{b-a+1}\left( \left( x^{1-a}+\alpha _{1}\left( 1-a\right) t\right)
^{\left( b-a+1\right) /\left( 1-a\right) }-x^{b-a+1}\right) .
\end{multline*}
In this case, $T_{x}$ has a shifted Weibull distribution, with mean 
\begin{equation*}
\mathbf{E}\left( T_{x}\right) =\frac{e^{\frac{\gamma _{1}}{b-a+1}x^{b-a+1}}}{%
\alpha _{1}\left( b-a+1\right) }\int_{x^{b-a+1}}^{\infty }u^{\frac{1-a}{b-a+1%
}-1}e^{-\frac{\gamma _{1}}{b-a+1}u}du<\infty .
\end{equation*}
\end{example}

%

\subsection{First jump time when $ +\infty $ is accessible}
If $x_{t}\left( x\right) $ reaches state $ + \infty $ in finite time $%
 I_\infty ( x) < \infty ,$ then we still have for all $ t \geq 0 $ the equality 
$$\mathbf{P}\left(
T_{x}>t\right) =e^{-\int_{0}^{t}\beta \left( x_{s}\left( x\right) \right)
ds} $$
which equals, for all $t < I_\infty ( x) , $ 
\begin{equation*}
 \mathbf{P}\left( T_{x}>t\right) =e^{-\int_{0}^{t}\beta \left( x_{s}\left(
x\right) \right) ds}
=e^{-\left[ \Gamma \left( x_{t}\left( x\right) \right) -\Gamma \left(
x\right) \right] } .
\end{equation*}
Letting $t\uparrow I_\infty ( x)  $ in the above equation, we get 
$$\mathbf{P}\left( T_{x}\geq I_\infty ( x)  \right) =e^{-\left[ \Gamma \left( \infty \right) -\Gamma \left(
x\right) \right] }$$ 
by monotone convergence, since $\Gamma $ is increasing, whence the necessary and sufficient condition 
\begin{equation}\label{eq:txexplosion}
\mathbf{P}\left( T_{x}\geq I_\infty ( x)  \right) =0\Longleftrightarrow
\Gamma \left( \infty \right) =\infty .
\end{equation}
Notice that under Assumption \ref{ass:gammainfty}, the representation \eqref{eq:txesp} remains valid for all $ x > 0 ,$ and also for $ x= 0 $ if $ 0$ is reflecting. Notice finally that $\mathbf{E}\left( T_{x}\right) < \infty $ since $T_x < I_\infty ( x)  $ almost surely.

\begin{example}
We consider  $\alpha \left( x\right) =\alpha _{1}x^{a}$ with $ a > 1 $ such that the solution $$x_{t}\left( x\right)
=\left( x^{1-a}+\alpha _{1}\left( 1-a\right) t\right) ^{1/\left( 1-a\right)
}$$ explodes in finite time at $I_\infty ( x)  =x^{1-a}/\left[
\alpha _{1}\left( a-1\right) \right] . $ Taking $\beta \left( x\right) =\beta
_{1}x^{b},$ we have for $ b \neq a-1,$ 
\begin{equation*}
\Gamma \left( x\right) =\frac{\gamma _{1}}{b-a+1}x^{b-a+1}\text{ such that  }\Gamma
\left( \infty \right) =\infty \Longleftrightarrow b>a-1.
\end{equation*}
If $ b>a-1$, then $ T_x < I_\infty ( x)  $ almost surely. 

If $0 < b<a-1,$ then  $\Gamma \left( \infty \right) =0 $ and $ \beta ( \infty ) = \infty, $ and $
T_{x}$ has an atom at $I_\infty ( x) $ with mass $e^{\frac{\gamma
_{1}x^{b-a+1}}{b-a+1}} .$ If $ b = 0, $ the process jumps at constant rate independently of its value (finite or infinite). Finally, if $ b < 0, $ then $\beta ( \infty ) = 0$ and $ T_x = + \infty $ with 
probability $e^{\frac{\gamma_{1}x^{b-a+1}}{b-a+1}} .$
\end{example}

\subsection{Joint distribution of $\left( T_{x},X_{T_{x}}\right) $}
Under the assumption $ I_\infty ( x) = + \infty ,$  we have for all $y\in \left[ 0,x_{t}\left( x\right) \right) ,$

\begin{eqnarray*}
\mathbf{P}\left( T_{x}\in dt,X_{T_{x}}\in dy\right) &=&dt\beta \left(
x_{t}\left( x\right) \right) e^{-\int_{0}^{t}\beta \left( x_{s}\left(
x\right) \right) ds}H\left( x_{t}\left( x\right) ,dy\right) \\
&=&dt\beta \left( x_{t}\left( x\right) \right) e^{-\int_{x}^{x_{t}\left(
x\right) }\gamma \left( z\right) dz}H\left( x_{t}\left( x\right) ,dy\right) .
\end{eqnarray*}
Moreover, 
$$
\mathbf{P}\left( T_{x}>\tau ,X_{T_{x}}\in dy\right) = e^{\Gamma \left( x\right) }\int_{x_{\tau }\left( x\right) }^{\infty
}dz\gamma \left( z\right) e^{-\Gamma \left( z\right) }H\left( z,dy\right) 
$$
and
$$
\mathbf{P}\left( X_{T_{x}}\in dy\right) 
=e^{ \Gamma ( x) } \int_{x}^{\infty }dz\gamma \left( z\right) e^{-\Gamma ( z) } H\left( z,dy\right) .$$

We close this section with an important remark that we shall rely on later.

\begin{proposition}\label{prop:strongfeller}
Under Assumptions \ref{ass:gammainfty} and \ref{ass:gamma0}, the jump chain $ (Z_k)_k, $ given by $ Z_k = X_{S_k}$ is strong Feller.
\end{proposition}

\begin{proof}
Let $g$ be any bounded and measurable function. Then $ x \mapsto \mathbf{E}_x ( g (Z_1) ) $ is continuous, since 
$$ \mathbf{E}_x ( g (Z_1) ) =  e^{ \Gamma ( x) } \int_{x}^{\infty }dz\gamma \left( z\right) e^{-\Gamma ( z) } \int_0^z  g(y) H(z, dy ),$$
which is continuous in $x$ by dominated convergence.
\end{proof}

\subsection{Classification of state $0$}
The classification of state $0$ is of utmost importance since return times to $0$ allow to decompose the trajectory of the process $(X_t)_t$ into independent excursions out of $0, $ implying recurrence of the process - under the condition that the process comes back to $0$ almost surely and is not stuck there. 

With $x>0$, state $0$ is non-absorbing or reflecting if and only if
\begin{equation*}
I_{0}\left( x\right) =\int_{0}^{x}\frac{dy}{\alpha \left( y\right) }<\infty .
\end{equation*}
If $I_{0}\left( x\right) =\infty , $ then
state $0$ is absorbing. 

$I_{0}\left( x\right) $ is the time necessary for $%
x_{t}$ to move from $0$ to $x>0$. In particular, if $I_{0}\left( x\right) <\infty ,$ then state $0$ is a reflecting boundary.  Moreover, if 
$I_{0}\left( x\right) =\infty ,$ then it is an absorbing boundary.

We can get IN from some $x\in \left( 0,\infty \right) $ to the boundary
point $0$ iff $H\left( x,0\right) >0$.

We can get OUT from the boundary point $0$ iff $I_{0}\left( x\right) <$ $%
\infty $ for some $x\in \left( 0,\infty \right) $.

This leads to four possible combinations for the boundary state $0$:
\begin{itemize}
\item
$H\left( x,0\right) >0$ and $I_{0}\left( x\right) <$ $\infty :$ regular
(accessible and reflecting).
\item
$H\left( x,0\right) >0$ and $I_{0}\left( x\right) =$ $\infty :$ exit
(accessible and absorbing).
\item
$H\left( x,0\right) =0$ and $I_{0}\left( x\right) <$ $\infty :$ entrance
(inaccessible and reflecting).
\item
$H\left( x,0\right) =0$ and $I_{0}\left( x\right) =$ $\infty :$ natural
(inaccessible and absorbing).
\end{itemize}
The first case is called regular because we can get in to $0$ and we can
start the process afresh from there. The second case is called exit because
we can get in to $0$ but cannot get out. The third is called an entrance
boundary because we cannot get in to $0$ but we can start the process there.
Finally, in the fourth case the process can neither get to nor start afresh
from $0$, so it is reasonable to exclude $0$ from the state space.

\subsection{Reaching state $\infty$ and explosion of the stochastic process}
As usual in the theory of jump processes, we say that the process possesses a finite explosion time $S_\infty $ if 
\begin{equation}\label{eq:explosion}
 \lim_{n \to \infty } S_n = S_\infty < \infty 
\end{equation} 
with positive probability, where $ S_1 < S_2 < \ldots$ is the sequence of successive jump times of the process. Explosion of the process $(X_t)_t$ therefore refers to the event that we 
observe an accumulation of an infinite number of jumps within a finite time interval. 

Clearly, $I_\infty ( x ) = \infty $ implies that the process does not explode in finite time. Indeed, the upper bound $ X_t \le x_t ( x) < \infty $ (if $ X_0 = x$) implies that the maximal jump rate of the process up to time $t$ is given by $ \max \{ \beta( x_s ( x) : s \le t \} $ which is finite by continuity of $ \beta .$ The following proposition discusses the case $ I_\infty ( x) < \infty.$ 

\begin{proposition}\label{prop:explos}
Suppose that $ \Gamma ( \infty ) = \infty $ and $I_{\infty }\left( x\right) < \infty $ for some (and hence all)  $x > 0.$ 
Let $ T_\infty ( x) = \inf \{ t > 0 : X_{t-}  = \infty |X_0 = x \} .$ Then 
$$ {\mathbf P} ( T_\infty ( x) < S_\infty  ) = 0.$$ 
\end{proposition}
The above result implies that the process is not able to reach the state $ +\infty $ before  the time of explosion $ S_\infty . $ 

\begin{proof}
Suppose that $ T_\infty ( x) <  S_\infty  $  with positive probability and write $ T = T_\infty ( x) .$ Let $ S_T = \sup \{ S_n : S_n < T \} $ be the last jump of the process strictly before hitting the state  $ +\infty .$   
$ T < S_\infty $ implies that there is only a finite number of jumps on $ [0, T] , $ such that, almost surely, $ S_T < T $ and $ X_{S_T} < \infty . $ Moreover, conditionally on $  X_{S_T}  = y < \infty , $ 
$$ X_{S_T+ t }  = x_t (y ) , \mbox{  for all } \;  t < T - S_T \mbox{ and }  T - S_T \stackrel{d}{=} I_\infty (y) .$$ 
In particular, $ X$ does not jump in $ ( S_T, T ) .$ However, since $ \Gamma ( \infty ) = \infty , $  by \eqref{eq:txexplosion}, almost surely, 
$$ T_y < I_\infty ( y ) , $$
implying that $ X$ does indeed jump strictly before time $T, $ which is a contradiction. 
\end{proof}

The above arguments show that on the event of explosion $\{ S_\infty < \infty \},$ the process approaches state $ \infty $ in finite time, that is, 
on $\{ S_\infty < \infty \}, $ we have that $\lim_{n \to \infty } X_{S_n} = \infty $ almost surely. This also follows from the following result which extends the classical explosion criterion for pure Markov jump processes without drift (see e.g. \cite{KK}) to the present frame of PDMP's. 

\begin{proposition}\label{prop:explos2}
Grant Assumptions \ref{ass:gammainfty} and \ref{ass:gamma0} and suppose moreover that $ I_0 ( x) < \infty .$ Then almost surely
$$ \left(S_\infty < \infty \right)  \; \Longleftrightarrow   \left( \sum_n  e^{\Gamma ( X_{S_n})} \int_{ X_{S_n} }^\infty \frac{1}{\alpha (z)}  e^{ - \Gamma (z) } dz  < \infty  \right).  $$ 
\end{proposition}

\begin{proof}
Let us write for short 
\begin{equation}\label{eq:e}
e(x) := \mathbf{E}\left( T_{x}\right)  = e^{\Gamma (x) } \int_x^\infty \frac{1}{\alpha (z)}  e^{ - \Gamma (z) } dz   .
\end{equation}
Then the process 
$$ A_n =  \sum_{k=1}^n \mathbf{E} \left( S_k - S_{k-1}  | {\mathcal F}_{S_{k-1}} \right) =\sum_{k=1}^n  e ( X_{S_{k-1}} ) $$
is the predictable increasing compensator of $S_n ,$ that is, $ M_n := S_n - A_n$ is a martingale. Putting $ \tau_a := \inf \{ n : A_{n+1} > a \} $ it follows that $ M^-_{n \wedge \tau_a } \le a , $ and the martingale convergence theorem implies that 
$ \{ A_\infty < \infty \} \subset \{ S_\infty < \infty \}$ almost surely. To prove the opposite inclusion, suppose $ S_\infty < \infty $ with positive probability. Then necessarily $ I_\infty ( x) < \infty .$ In particular, recalling \eqref{eq:txexplosion},
$$ \sup_n (S_n - S_{n-1} ) \le \sup_n I_\infty ( X_{S_{n-1}} ) \le \int_0^\infty \frac{1}{\alpha (y ) } dy < \infty $$ 
since $ 0 $ is reflecting by assumption and since $ I_\infty (x) < \infty .$  Introducing  the stopping time $ \sigma_a := \inf \{ n : S_n > a \} ,$ it follows from the above that $ \sup_n \mathbf{E} (  M^+_{n \wedge \sigma_a } ) < \infty. $ Classical arguments then allow to conclude that $ \{ S_\infty < \infty \} \subset \{ A_\infty < \infty \} $ almost surely. 
\end{proof}

In what follows, we give conditions ensuring that the process reaches state $ +\infty $ starting from any point $x\in(0,\infty).$ We also exhibit conditions implying that the process comes down from infinity to $y\in(0,\infty)$.

We can get IN from some $x\in \left( 0,\infty \right) $ to the boundary
point $\infty $ iff $\Gamma ( \infty ) <\infty $ and $ I_\infty (x) < \infty.$ 

We can get OUT from the boundary point $\infty $ iff $H\left( \infty
,y\right) >0$ for some $y\in \left( 0,\infty \right) .$ 

This leads to the following relevant combinations for the boundary state $\infty .$ To classify them, we introduce $\Sigma=\sum_{n\geq 1} e (X_{S_n}), $ where $X_{S_n} $ is the embedded chain of $X. $ Then we have:

$\Sigma  <\infty $ and $H\left( \infty ,y\right) >0:$ regular (accessible and
reflecting).

$\Sigma < \infty  $ and $H\left( \infty ,y\right) =0:$ exit (accessible and
absorbing).

%

\section{Regularity of the transition operator and speed measure}
We describe the infinitesimal generators of the process $X.$

\textbf{Backward:} With $u_{t}\left( x\right) :=\mathbf{E}_{x}u\left( X_{t}\right) $, $%
u_{0}\left( x\right) =u\left( x\right) $, we have (Kolmogorov backward
equation) 
\begin{equation*}
\partial _{t}u_{t}\left( x\right) =\left( Gu_{t}\right) \left( x\right) ,
\end{equation*}
where $G$ is given in \eqref{eq:G}.

\textbf{Forward:} With $\Pi _{t,x}\left( dy\right) =\mathbf{P}_{x}\left(
X_{t}\in dy\right) $, $\Pi _{0,x}\left( dy\right) =\delta _{x}$, this also
means 
\begin{equation*}
\frac{d}{dt} \int_{0}^{\infty }u\left( y\right) \Pi _{t,x}\left( dy\right)
=\int_{0}^{\infty }\left( Gu\right) \left( y\right) \Pi _{t,x}\left(
dy\right) .
\end{equation*}
Notice that the measure $\Pi _{t,x}\left( dy\right) $ has support $%
\left[ 0,x_{t}\left( x\right) \right] $ with an atom at $x_{t}\left(
x\right) $ with mass $\mathbf{P}\left( T_{x}>t\right) . $ 
Considering the family of test functions $u\left( y\right) =e_{\lambda
}\left( y\right) :=e^{-\lambda y}$, $\lambda \geq 0$, for which 
\begin{equation*}
\left( Ge_{\lambda }\right) \left( x\right) =-\lambda \alpha \left( x\right)
e_{\lambda }\left( x\right)  +\lambda \beta \left( x\right)
\int_{0}^{x} {H}\left( x,y\right) e_{\lambda }\left( y\right) dy,
\end{equation*}
we get, using Fubini's theorem and putting $\Pi _{t,x}\left( y\right)
=\int_{0}^{y}\Pi _{t,x}\left( dz\right) ,$ 
\begin{multline}\label{eq:prel}
\frac{d}{dt} \int_{0}^{\infty }dye_{\lambda }\left(
y\right) \Pi _{t,x}\left( y\right)= \frac{d}{dt}\frac{1}{\lambda }\int_{0}^{\infty }e_{\lambda }\left( y\right) \Pi _{t,x}\left( dy\right) \\
=-\int_{0}^{\infty }e_{\lambda }\left( y\right) \alpha \left( y\right) \Pi
_{t,x}\left( dy\right)  +\int_{0}^{\infty }dye_{\lambda }\left( y\right) \int_{y}^{\infty
}\beta \left( z\right) {H}\left( z,y\right) \Pi _{t,x}\left(
dz\right).
\end{multline}

Writing $ {\mathcal D}_+' ( \R)$ for all distributions having support in $[0, \infty ) , $  we define the distribution $ \delta_t \Pi_{t, x } $ by 
$$ < \delta_t \Pi_{t, x }  , u > := \frac{d}{dt} \int u ( y) \Pi_{t, x} (y) dy $$
for any smooth test function $u$ having compact support. Notice that $\delta_{t}\Pi _{t,x} $ is of compact support.

Therefore, Laplace transforms characterizing distributions with compact support in $ \R_+,
$ by
duality (Kolmogorov forward equation) 
\begin{equation}\label{eq:forward}
\delta _{t}\Pi _{t,x} =-\alpha \left( y\right) \Pi
_{t,x}\left( dy\right) +dy\int_{y}^{\infty }\beta \left( z\right) H\left(
z,y\right) \Pi _{t,x}\left( dz\right) .
\end{equation}

\begin{proposition}\label{prop:2}
Suppose either that $ \alpha $ is strictly positive on $ [0, \infty )$ or, in case that $\alpha (0) = 0 ,$ either that $ I_0(x) < \infty  $ or that $ H(x, 0 ) = 0 $ for all $ x > 0.$ Then for all $ x > 0,$  $\Pi_{t, x } $ is absolutely continuous on $ [0, x_t( x) ) .$  
\end{proposition}

\begin{proof}
Let $g$ be a smooth test function having compact support in $ [0, x_t(x) ) .$ Then $ \mathbf{E}_x ( g (X_t)) = \mathbf{E}_x ( g(X_t) \mathbf{1}_{\{ t \le T_x \} })  .$ Recall that $ S_1 < S_2 < \ldots $ denote the successive jumps of $ X_t .$ Then we have 
$$   \mathbf{E}_x ( g (X_t )) = \sum_{n=1}^\infty \mathbf{E}_x ( g(X_t ) \mathbf{1}_{\{ N_t = n  \} }).$$
The joint law of $ Y_n := (S_1, \ldots, S_{n+1}, X_{S_1}  , \ldots, X_{S_n}  ) $ under $\mathbf{P}_x$ is given by 
\begin{multline*}
f_Y ( s_1, \ldots , s_{n+1}, dx_1, \ldots, d x_n ) ds_1 \ldots d s_{n+1} =   \beta ( x_{s_1} (x) ) e_{s_1} ( x) ds_1 \\ \int_{\R_+}  H ( x_{s_1} ( x) , d x_1 )  
\beta ( x_{s_2} (x_1) ) e_{s_2} ( x_1 ) ds_2 \ldots \\
  \int_{\R_+}H ( x_{s_n} ( x_{n-1}) , d x_n )   \beta ( x_{s_{n+1}} (x_n) ) e_{s_{n+1}} ( x_n) d s_{n+1} ,
\end{multline*}  
where 
$$ e_t ( x) := e^{- \int_0^t \beta ( x_s ( x) ) ds } .$$
Therefore, 
$$
 \mathbf{E}_x ( g(X_t  )\mathbf{1}_{\{ N_t = n  \} }) 
= 
 \int_{ [0, t ]^n \times [ t , \infty [ } \int_{ \R_+^n }  f_Y ( s_1, \ldots , s_{n+1}, dx_1, \ldots, d x_n )  g ( x_{t - s_n} (x_n)) ds_1 \ldots d s_{n+1} .
$$
Notice that under our condition, $ x_{t- s_n} ( x_n) > 0 $ for all $ s_n < t . $ In particular we also have that $ \alpha (x_{t- s_n} ( x_n)) > 0.$ Using the change of variables $ s_n \mapsto z (s_n)$ with $ z (s_n) :=x_{t-s_{n}}(x_{n}) \in [ x_n, x_t ( x_n) ],$ for fixed $x_{n},$
with $s_{n}=z^{-1}(z,x_{n}),$ we then have 
$$ \frac{dz}{d s_n} = - \alpha ( x_{t - s_n} (x_n) ) = - \alpha ( z) , $$ such that 
\begin{multline*}
  \mathbf{E}_x ( g(X_t \mathbf{1}_{\{ N_t = n  \} })   =\int_{ \R_+} dz \frac{ g(z)}{ \alpha ( z) }
   \Big( \int_{ [ 0, t]^{n- 1} \times [t, \infty [} \int_{ \R_+^n } \\
\mathbf{1}_{ \{ x_n \le z \le x_t (x_n) \}}  f_Y ( s_1, \ldots , z^{- 1 } ( z, x_n) ,  s_{n+1}, dx_1, \ldots, d x_n ) d s_1 \ldots ds_{n- 1 } d s_{n+1}  \Big) . 
\end{multline*}
Summing over $n $ implies the result. 
\end{proof}

Let us come back to equation (\ref{eq:forward}) together with the preceding considerations. We now know that under the conditions of Proposition \ref{prop:2}, $ \Pi_{t, x} ( dy ) $ admits a density $\pi
_{t,x}\left( y\right) $ on $\left[ 0,x_{t}\left( x\right) \right) $ and we have 
\begin{equation*}
\Pi _{t,x}\left( dy\right) =\mathbf{P}\left( T_{x}>t\right) \delta
_{x_{t}\left( x\right) }\left( dy\right) + \pi _{t,x}\left( y\right) \mathbf{1}_{\left( y\in \left[
0,x_{t}\left( x\right) \right) \right) }dy .
\end{equation*}

\eqref{eq:forward} implies that on $ [0, x_t (x) ) , $ the distribution $\delta_{t}\Pi _{t,x} $ has 
a density $ \delta_{t}\Pi _{t,x} (y) $ given by 
\begin{eqnarray*}
\delta_{t}\Pi _{t,x} (y)  &=&-\alpha \left( y\right) \pi
_{t,x}\left( y\right) +\int_{y}^{\infty }\beta \left( z\right) H\left(
z,y\right) \Pi _{t,x}\left( dz\right)\\
&=&-\alpha \left( y\right) \pi
_{t,x}\left( y\right) +\int_{y}^{\infty }\beta \left( z\right) H\left(
z,y\right) \pi _{t,x}\left( z\right) dz + \beta ( x_t (x) ) H( x_t( x), y ) \mathbf{P}\left( T_{x}>t\right).
\end{eqnarray*}
In the separable case $ H(x, y ) = h(y)/ h (x), $ this can be rewritten as 
$$
\delta_{t}\Pi _{t,x} (y) 
=  -\alpha \left( y\right) \pi _{t,x}\left( y\right)
+h\left( y\right) \int_{y}^{\infty }\frac{\beta \left( z\right) }{h\left(
z\right) }\pi _{t,x}\left( z\right) dz
 + \beta ( x_t (x) ) \frac{h(y)}{h( x_t( x) )} \mathbf{P}\left( T_{x}>t\right).
$$

If $\widetilde{\pi }_{t,x}\left( y\right) :=\alpha \left( y\right) \pi
_{t,x}\left( y\right) $,  recalling that 
$
\gamma ( x) := \beta(x) / \alpha (x) , 
$ 
we have for all  $y \in [0, x_t (x) )  ,$ 
$$
\delta_{t}\Pi _{t,x}\left( y\right) =-\widetilde{\pi }_{t,x}\left(
y\right) +\int_{y}^{\infty }\gamma \left( z\right) H\left( z,y\right) 
\widetilde{\pi }_{t,x}\left( z\right) dz \\
 + \beta ( x_t (x) ) H( x_t( x), y ) \mathbf{P}\left( T_{x}>t\right) 
$$
In the separable case, this reads 
$$
\delta_{t}\Pi _{t,x}\left( y\right) {=} -\widetilde{\pi }%
_{t,x}\left( y\right) +h\left( y\right) \int_{y}^{\infty }\frac{\gamma
\left( z\right) }{h\left( z\right) }\widetilde{\pi }_{t,x}\left( z\right) dz ,
+ \beta ( x_t (x) ) \frac{h(y)}{h( x_t( x) )} \mathbf{P}\left( T_{x}>t\right).
$$
Clearly, under the conditions of Proposition \ref{prop:2},  $\Pi _{t,x}\left( 0\right) =0.$ We conclude for $y=0$: if $h\left(
0\right) =0$, $\widetilde{\pi }_{t,x}\left( 0\right) =0.$ If $h\left(
0\right) >0,$ then 
\begin{equation*}
\widetilde{\pi }_{t,x}\left( 0\right) =h\left( 0\right) \int_{0}^{\infty }%
\frac{\gamma \left( z\right) }{h\left( z\right) }\widetilde{\pi }%
_{t,x}\left( z\right) dz + \beta ( x_t (x) ) \frac{h(0)}{h( x_t( x) )} \mathbf{P}\left( T_{x}>t\right),
\end{equation*}
and the value of $\widetilde{\pi }_{t,x}\left( 0\right) $ 
requires the knowledge of the whole $\widetilde{\pi }_{t,x}\left( z\right) ,$ for all $ z \in ( 0, x_t (x) ).$

%
%
%

We close this section with the following observation.

\begin{proposition}
Suppose that $I_\infty ( x) < \infty $ and that $ \mathbf{P} ( T_x < I_\infty ( x) ) = 1.$ Grant moreover the assumptions of Proposition \ref{prop:2}. Then $\Pi_{t, x} $ is absolutely continuous on $ \R_+ $ for all $ t \geq  I_\infty ( x) .$ 
\end{proposition}

Whenever an invariant measure $\pi  $ exists which is not equal to $ \delta_0,$ the same argument leading to \eqref{eq:prel} implies that 
$ \alpha (x) \pi (dx)  $ admits a Lebesgue density $ \tilde \pi ( x) $ solving the functional equation 
\begin{equation*}
\widetilde{\pi }\left( y\right) =\int_{y}^{\infty }\gamma \left( z\right)
H\left( z,y\right) \widetilde{\pi }\left( z\right) dz
\end{equation*}
for $\lambda-$almost all $ y > 0.$ In the separable case $H\left( z,y\right) = \frac{h\left( y\right) }{h\left(
z\right) }$, this yields the explicit expression 
\begin{equation}\label{eq:speed}
\pi \left( y\right) =C \frac{h\left( y\right) }{\alpha \left( y\right) }%
e^{-\Gamma \left( y\right) },
\end{equation}
up to a multiplicative constant $ C > 0.$ Notice that under Assumption \ref{ass:gamma0}, $ \pi $ is integrable in $ 0+$ if and only if $ \int_0 h(x) /\alpha (x) dx < \infty $ which is equivalent to $0$ reflecting in case $ h(0) > 0.$

\section{Harris Recurrence}

\subsection{Recurrence of $X$ and of the embedded chain}
In what follows we shall rely on the notion of Harris recurrence for Markov processes which we recall here for the convenience of the reader. 

\begin{definition}[see \cite{ADR}]
$X$ is called {\em Harris recurrent} if there exists some $\sigma$-finite measure $m$ on $(\R_+,{\mathcal B} (\R_+))$ such that for all $ A \in {\mathcal B} (\R_+),$ 
$$m(A)>0 \;   \mbox{ implies}  \quad P_x\left(\int^\infty_0 1_A(X_s)ds=\infty\right)=1  \mbox{ for all }  x\in \R_+ .$$
\end{definition}

If is well-known (see again \cite{ADR}) that if $X$ is Harris recurrent, then there is a unique (up to constant multiples) invariant measure 
$\pi$ for $X$, and the above property holds with $\pi$ in place of $m$. $X$ is then called {\em positive recurrent} 
(or also sometimes  {\em ergodic}) if 
$\pi(\R_+ )<\infty$,  {\em null recurrent} if $\pi (\R_+)=\infty$.

Notice that whenever our process is Harris with invariant measure $ \pi \neq \delta_0, $ then its explicit density is necessarily given by \eqref{eq:speed} (in the separable case).

\begin{example}
If $h(x) \sim e^{\Gamma(x)} $ as $ x \to \infty, $ we have $\pi(x)\sim \frac{1}{\alpha(x)},$ as $ x \to \infty.$ In particular, $ \int^\infty \pi (y ) dy < \infty $ if and only if $ I_\infty ( x) < \infty $ for some (and thus all) $ x > 0. $ This means that the deterministic flow hits state $ + \infty$  in finite time. Thus, the fact that the deterministic flow
hits state $\infty $ in finite time helps the process being positive recurrent (compare also to \eqref{eq:txexplosion}).
\end{example}

%
%
%
Let us now come back to our general framework. The following result establishes a relation between $ \pi$ and the invariant measure of the jump chains $(U_k)_k $ and $ (Z_k)_k$ where $ U_k = X_{S_k - } $ and $ Z_k = X_{S_k}, $ with $( S_k)_{k \geq 1 } $ the sequence of successive jump times of the process. 
\begin{proposition}
Suppose that $ X$ is Harris recurrent having invariant measure $ \pi $ such that $ 0 < \pi (\beta  ) < \infty .$ Then $ (U_k)_k $ and $ (Z_k)_k$ are both Harris recurrent. Their  invariant measures $ \pi^U $ and $\pi^Z$ are respectively given by 
$$ \pi^U ( g) = \frac{1}{\pi (\beta ) } \pi  ( \beta g ) , \;  \pi^Z ( g) = \frac{1}{\pi (\beta ) } \pi  ( \beta H g ) , $$
for any $g : \R^N \to \R$ measurable and bounded, where 
$$ \beta H g (x) = \beta ( x) \int H (x, dy ) g (y) .$$ 
\end{proposition}

\begin{proof}We just give the proof for $ (Z_k)_k, $ the case of $ (U_k)_k $ is treated analogously.
Let $ g \geq 0 $ be a bounded positive test function. It is sufficient to prove that 
$ \frac{1}{n} \sum_{k=1}^n g ( Z_k) \to \pi^Z ( g) $
as $n \to \infty, $ $P_x-$almost surely, for any fixed starting point $x .$ But
$$ \frac{1}{n} \sum_{k=1}^n g ( Z_k) = \frac{1}{n} \sum_{k=1}^n g ( X_{S_k}).$$
Introduce the jump measure 
$$ \mu ( ds, dy , dz  ) = \sum_{n \geq 1 } 1_{ \{  S_n < \infty \} }  \delta_{ (S_n, X_{S_n - } , X_{S_n} ) }  (dt, dy, dz).$$
Its compensator is given by 
$$\nu  (ds, dy, dz) = \beta ( X_{s-} ) ds \delta_{X_{s-} } (dy) \int H (y, dz) .$$ 

Putting $ N_t = \sup \{ n : S_n \le t \},$  
$$
 \lim_{ n \to \infty }  \frac{1}{n} \sum_{k=1}^n g ( X_{S_k}) = \lim_{t \to \infty}\frac{t}{N_t} \frac{1}{t} \sum_{k=1}^{N_t} g ( X_{S_k}) 
= \lim_{t \to \infty}\frac{t}{N_t} \frac{A_t}{t}  , $$
where $A_t = \int_0^t\int_{\R^N}\int_{\R^N} g (z) \mu ( ds, dy, dz) $ and $N_t$ are additive functionals of the process $X.$ 
By the ergodic theorem for the process $X$ (which holds thanks to the Harris recurrence of $ X_t$), $ N_t/t \to E_\pi ( N_1) $ and $ A_t/t \to E_\pi ( A_1 ) ,$ and this convergence holds almost surely, for every starting point $x.$ But $ E_\pi (N_1) = E_\pi ( \hat N_1) $ and $E_\pi ( A_1) = E_\pi (\hat A_1) , $ where 
$$
\hat N_t = \int_0^t \int \int \nu  (ds, dy, dz) = \int_0^t \beta  (X_s) ds $$
and
$$   \hat A_t = \int_0^t \int \int g ( z) \nu  (ds, dy, dz)=  \int_0^t \beta ( X_s) \int H(X_s, dz ) g(z) ds = \int_0^t \beta H g ( X_s) ds   .
$$
Therefore, $ E_\pi ( N_1) = \pi( \beta ) $ and $ E_\pi ( A_1) = \pi ( \beta H g) ,$
and this finishes the proof.
\end{proof}

We use the above considerations to discuss rapidly that explosion of the process $X$ (in the sense that $ S_\infty < \infty $ with positive probability) is only possible if the jump chain $Z_n$ is transient.

\begin{proposition}\label{prop:7}
If $Z_n$ is recurrent, explosion of $X_t $ (that is, $ \lim S_n = S_\infty  < \infty $ with positive probability) is not possible. 
\end{proposition}

\begin{proof}
If $ \pi^Z =  \delta_0, $ then non-explosion of the continuous time process is trivially implied. Let us therefore suppose that $ \pi^Z \neq 0.$  We know that explosion of $X $ is equivalent to $ \sum_{n \geq  1}e ( Z_n)  < \infty $ (recall the definition of $e$ in \eqref{eq:e}). But, if $Z_n$ is recurrent (possibly null-recurrent), we know that for any function $g > 0 $ such that $ \pi^Z (g) \in (0, \infty ), $ 
$$ \frac{ \sum_{k=1}^n e( Z_k) }{ \sum_{k=1}^n g ( Z_k) }  \to \pi^Z ( e )/ \pi^Z ( g)  $$
almost surely. Since $ \sum_{k=1}^n g ( Z_k) \uparrow \infty $ as $ n \to \infty, $ explosion implies that $\pi^Z ( e ) = 0 ,$ whence $ e = 0  $ $\pi^Z -$almost surely. $e  $ being strictly positive on $(0, \infty ) , $ this yields a contradiction. 
\end{proof} 

\begin{corollary}
In particular, if $ Z_n$ is recurrent (positive or null), then $ X$ is also recurrent (positive or null).
\end{corollary}

\begin{proof}
$Z_n$ recurrent implies $ S_n \uparrow \infty $ almost surely, thanks to Proposition \ref{prop:7}. Now let $ A \in {\mathcal B} ( \R_+) $ be such that $\pi^Z ( A) > 0 $ implying that $ 1_A ( Z_n) = 1 $ infinitely often. Then $ \limsup_{t \to \infty } 1_A ( X_t ) \geq \limsup_{n \to \infty } 1_A ( X_{S_n } ) = \limsup_{n \to \infty } 1_A ( Z_n) = 1 , $ whence the recurrence of $ X_t.$ 
\end{proof}

\subsection{Exit probabilities and excursions}
In this section we propose a thorough study of the return times to $0 $ that enable us to state sufficient conditions for positive recurrence. Throughout this section we impose Assumptions \ref{ass:gammainfty} and \ref{ass:gamma0}. With $x>0$, we introduce
\begin{equation*}
\tau _{x,0}=\inf \left\{ t>0:X_{t}=0\mid X_{0}=x\right\} 
\end{equation*}
the first time the process comes back to $0.$
%
%
%
%

In what follows we fix $ 0 < x < b$ and are interested in establishing explicit formulae for 
$$p(x,b) = \mathbf{P}_x\left( \tau _{x,0}<\tau _{x,b}\right) .$$ 
Notice that it follows from the properties of our process that $ \lim_{ x \to b } p(x,b) = p(b,b) = 0.$ However, we do not have that $\lim_{x \to 0 } p(x,b) = p (0,b ) = 1 .$ In general, $ p ( 0 ,b) < 1 $ is related to the height of an excursion between two successive visits to $0,$ see below. 

A first step analysis implies that 
$$ p (x,b) = \int_0^{I_x( b ) } {\mathcal L} ( T_x) (ds) \left( H( x_s ( x) , 0 ) +  \int_{0+}^{x_s ( x) }  H( x_s ( x), dy ) p (y,b ) \right)   ,$$ 
with $ I_x(b ) = \int_x^b \frac{dy}{ \alpha (y )} $ the time needed to go from $x$ to $b.$ 
A simple change of variables implies that 
$$ p( x,b) = \int_x^b \gamma ( v) e^{ - (\Gamma ( v) - \Gamma ( x) ) } H( v, 0 ) dv +\int_x^b \gamma ( v) e^{ - (\Gamma ( v) - \Gamma ( x) ) }  \int_{0+}^v  H(v, dy ) p(y,b) .$$ 
In the sequel we shall only consider the separable case $H\left( x,y\right) =\frac{h\left( y\right) }{%
h\left( x\right) }$ with 
$$ h(0) > 0.$$ 
In this case, the above formula implies that $x \mapsto p(x,b)  \in C^1 ( [0, b ] ) .$ 
Recalling that $ p( b,b ) = 0, $ we rewrite 
$$ p (y,b ) = - \int_y^b p' (z,b) dz ,$$
where $p' (x, b ) = \partial_x p (x, b ) $ denotes partial derivative with respect to the initial position.  We obtain
\begin{multline*}
p(x,b) =(1 - p(0,b) )  \int_x^b \! \gamma ( v) e^{ - (\Gamma ( v) - \Gamma ( x) ) } \frac{h(0)}{h(v)}  dv +  \int_x^b \!  \gamma ( v) e^{ - (\Gamma ( v) - \Gamma ( x) ) } p(v,b) dv \\
- \int_x^b \frac{\gamma (v) }{h(v)} e^{ - (\Gamma ( v) - \Gamma ( x) ) } \int_0^v h(z) p' (z,b) dz dv .
\end{multline*}
Taking derivatives, we obtain 
$$ p'(x,b) h (x) = \gamma ( x) \int_0 ^x h(z)  p'(z,b) dz - \gamma ( x) (1 - p(0,b) ) h(0) . $$
Let 
$$ \kappa ( x) := \int_0 ^x h(z)  p'(z,b) dz - (1 - p(0,b )) h (0) , 0 \le x \le b , $$ 
then we have $ \kappa ' (x) = h (x) p' (x,b) $ and $ \kappa (0) =- (1 - p(0,b )) h (0) .$ The above equation reads 
$$ \kappa' (x) = \gamma ( x) \kappa ( x) $$  leading to 
$$ \kappa ( x) = C e^{\Gamma( x) }, $$ 
where we choose $ \Gamma $ such that $ \Gamma (0) = 0 $ and where $C$ is such that 
$$C  = - (1 - p(0,b )) h (0)  ; \mbox{ that is, } C = -  h(0) ( 1 - p(0,b ) ).$$ 
We deduce from this that 
$$
p' (x,b)   = C\frac{\gamma \left( x\right) }{h\left(
x\right) }e^{\Gamma \left( x\right) } ,$$
and thus, using once more that $ p(b,b) = 0,$ 
$$ 
p \left( x,b \right)  =- C  \int_{x}^b \frac{\gamma \left( y\right) }{h\left(
y\right) }e^{\Gamma \left( y\right) }dy=    h(0) (1 - p(0 ,b))  \int_{x}^b \frac{\gamma \left( y\right) }{h\left(
y\right) }e^{\Gamma \left( y\right) }dy . 
$$
Finally, the value of $ p(0,b) $ is deduced from 
$$ p(0,b) = h(0) (1 - p(0,b ))  \int_{0}^b \frac{\gamma \left( y\right) }{h\left(
y\right) }e^{\Gamma \left( y\right) }dy  .$$
Let 
\begin{equation}\label{eq:s}
s(x) = \int_0^x \frac{\gamma \left( y\right) }{h\left(
y\right) }e^{\Gamma \left( y\right) }dy , \, \Gamma (y) = \int_0^y \gamma ( t) dt.
\end{equation}
Notice that under Assumption \ref{ass:gamma0} and supposing that $ h(0) > 0, $ $ s(x) $ is well-defined for any $x \geq 0.$ 

We obtain 
\begin{equation}\label{eq:p0h}
 p(0 ,b) = \frac{h(0)s(b)}{1+h(0)s(b)} \mbox{ and } \mathbf{P}\left( \tau _{x,0}<  \tau _{x,b}\right) =  p(0, b ) [ 1 - \frac{ s(x)  }{ s(b) }] .
\end{equation} 


We have just proven the following
\begin{theorem}\label{prop:scale} 
Grant Assumptions \ref{ass:gammainfty} and \ref{ass:gamma0} and let $ 0 < x < b .$  Suppose moreover that $H\left( x,y\right) =\frac{h\left( y\right) }{%
h\left( x\right) }$ with $ h(0 ) > 0. $ Suppose that $ \Gamma (0) = 0 $ and put 
$$ \kappa := 1/h(0) .$$ 
Then 
\begin{equation}\label{eq:xh0}
  \mathbf{P}\left( \tau _{x,0}> \tau _{x,b}\right)   = \frac{ \kappa + s( x) }{\kappa + s(b) }. \end{equation}
\end{theorem} 

Notice that in case  $h\left( x\right) =1$ (total disaster), we obtain
\begin{equation*}
\mathbf{P}( \tau _{x,b}<\tau _{x,0}) =e^{- (\Gamma(h)- \Gamma(x))} . 
\end{equation*}

{\bf Discussion of the role of $0.$} Theorem \ref{prop:scale} holds true in both cases $0$ reflecting or absorbing. However what follows does only make sense in case  $0$ is reflecting, that is, $ I_0 (x) < \infty. $ In this case we may introduce the height $\mathcal H$ of an excursion by
$$ \mathcal H = \sup \{ X_t  : t < \tau_{0, 0} | X_0 = 0  \} ,$$
where $ \tau_{0,0} = \inf \{ t > 0 : X_t = 0  \}  > 0 $ is the first return time to $0.$ Since $  \tau_{x, 0 } \stackrel{\mathcal L}{\to } \tau_{0,  0 } $ as $ x \to 0  ,$ we may interpret $ p (0, b ) $ by means of the distribution function of the height of an excursion. 

\begin{proposition} 
Grant the assumptions of Theorem \ref{prop:scale} and suppose that $ I_0 ( x) < \infty .$ Then 
\begin{equation}\label{eq:H}
{\mathbf P}(\mathcal H < b ) = {\mathbf P} ( \tau_{0,0} < \tau_{0, b } )   = p(0, b ) = \frac{ s(b) }{ \kappa + s(b ) } .
\end{equation}
\end{proposition}


The successive return times of the process $X$ to $0$ induce a basic regeneration scheme and are thus related to the recurrence of the process. 

\begin{proposition}\label{cor:14}
Grant Assumptions  \ref{ass:gammainfty} and \ref{ass:gamma0} and suppose moreover that $H\left( x,y\right) =\frac{h\left( y\right) }{%
h\left( x\right) }$ with $ h(0 ) > 0,$ that $I_\infty ( x) = \infty $ and $I_0(x)<\infty .$ Then the process is recurrent if and only if $ s( \infty ) = \infty , $ where the function $ s (x)$ is given by  \eqref{eq:s}. In this latter case, $ \tau_{x, 0 } < \infty $ almost surely, and  the unique invariant measure possesses a Lebesgue density on $ \R_+$ which is given by \eqref{eq:speed}. The process is positive recurrent if $\int^\infty \frac{h(x)}{\alpha ( x) } e^{ - \Gamma ( x)} dx < \infty ,$ null-recurrent else. 
\end{proposition}

\begin{proof}
Suppose $ s( \infty ) = \infty .$ 
We let $ b \to \infty $ in \eqref{eq:p0h} and notice that $ \lim_{b \to \infty } p(0, b ) = 1 $ such that
$$ \mathbf{P}\left( \tau _{x,0}<\tau _{x, \infty }\right)  = 1.$$

This implies that $\tau _{x,0} < \infty $ almost surely. 

On the other hand, suppose that the process is recurrent. It is straightforward to show that the recurrence implies that $ \tau_{0, 0 } < \infty $ almost surely (recall that $ 0$ is reflecting by assumption and that $ \beta $ is positive on $ (0, \infty) .$)  
Since $ \mathcal H \le x_{\tau_{0,0}}(0) $ and since $ I_\infty (x) = \infty, $ this implies that $ \mathcal H < \infty $ almost surely, i.e., $ \lim_{b \to \infty } {\mathbf P} ( \mathcal H < b ) =\lim_{b \to \infty}  p(0, b ) = 1 .$ Under our assumptions, this is only possible if $ s( \infty ) = \infty,$ since $ \kappa \neq 0.$   
\end{proof}


\begin{remark}
We impose all assumptions of proposition \ref{cor:14} except that now we consider the absorbing case $ I_0 (x) = \infty .$ In this case we still have that $ \tau_{x, 0 } < \infty $ almost surely if and only if $ s( \infty ) = \infty: $ the process gets
absorbed in $ 0 $ after a finite time almost surely and then stays there forever.
\end{remark}

When $h\left( x\right) =1$ (total disasters), the event $\tau _{x,b}<\tau
_{x,0}$ coincides with the event $T_{x}>I_{b}\left( x\right) $ where $%
I_{b}\left( x\right) =\int_{x}^{b}dy/\alpha \left( y\right) $ is the time
needed for the flow to reach level $b$ starting from $x$.

\begin{example}
Consider a growth model with $\alpha \left( x\right) =\alpha _{1}x^{a},$\ $%
\beta \left( x\right) =\beta _{1}$, $\gamma \left( x\right) =\gamma
_{1}x^{-a}$ and assume $h\left( x\right) =1.$\ Assuming\ $a<1$\ for which
boundary $0$\ is reflecting, then 
\[
x_{t}\left( x\right) =\left( x^{1-a}+\alpha _{1}\left( 1-a\right) t\right)
^{1/\left( 1-a\right) }=b\Rightarrow I_{b}\left( x\right) =\frac{%
b^{1-a}-x^{1-a}}{\alpha _{1}\left( 1-a\right) }.
\]
Thus, 
\begin{multline*}
{\bf P}\left( \tau _{x,b}<\tau _{x,0}\right) ={\bf P}\left(
T_{x}>I_{b}\left( x\right) \right) 
\\
={\bf P}\left( T_{x}>\frac{b^{1-a}-x^{1-a}}{\alpha _{1}\left( 1-a\right) }%
\right) =e^{-\left[ \Gamma \left( x_{t}\left( x\right) \right) -\Gamma
\left( x\right) \right] }\mid _{t=\frac{b^{1-a}-x^{1-a}}{\alpha _{1}\left(
1-a\right) }}=\frac{e^{\Gamma \left( x\right) }}{e^{\Gamma \left( b\right) }}
\end{multline*}
with $\Gamma \left( x\right) =\frac{\gamma _{1}}{1-a}x^{1-a}$. As $%
x\rightarrow 0$, 
\[
{\bf P}\left( \tau _{0,b}<\tau _{0,0}\right) ={\bf P}\left( \mathcal H\geq b\right) =%
{\bf P}\left( T_{0}>\frac{b^{1-a}}{\alpha _{1}\left( 1-a\right) }%
=I_0 ( b)  \right) =e^{-\Gamma \left( b \right) },
\]
where $\mathcal H$\ denotes the height of an excursion, which makes sense because
boundary $0$\ is reflecting and the chain is recurrent ($s\left( \infty
\right) =\infty $). So here 
\[
\mathcal H\stackrel{d}{=}\left( \alpha _{1}\left( 1-a\right) T_{0}\right) ^{1/\left(
1-a\right) },
\]
showing how height and length of excursions scale.
\end{example}

\begin{example}
Consider a growth model with $\alpha \left( x\right) =\alpha _{0}+\alpha
_{1}x$ (Malthus growth with immigration)$,$\ $\beta \left( x\right) =\beta
_{1}$, $\gamma \left( x\right) =\beta _{1}/\left( \alpha _{0}+\alpha
_{1}x\right) $ and assume $h\left( x\right) =e^{x}.$ We have 
\[
\Gamma \left( x\right) =\frac{\beta _{1}}{\alpha _{1}}\log \left( \alpha
_{0}+\alpha _{1}x\right) 
\]
satisfying Assumptions $1$ and $2$. State $0$ is reflecting and the process $%
X$ is transient at $\infty $. Here $\kappa =\frac{\beta _{1}}{\alpha _{1}}%
\log \alpha _{0}$, and 
\[
s\left( x\right) =\beta _{1}\int_{0}^{x}\left( \alpha _{0}+\alpha
_{1}y\right) ^{\beta _{1}/\alpha _{1}-1}e^{-y}dy=\frac{\beta _{1}e^{\alpha
_{0}/\alpha _{1}}}{\alpha _{1}}\int_{\alpha _{0}}^{\alpha _{0}+\alpha
_{1}x}z^{\beta _{1}/\alpha _{1}-1}e^{-z/\alpha _{1}}dz,
\]
involving an integral Gamma function. It holds that 
\[
{\bf P}\left(\mathcal  H\geq b\right) =\frac{\kappa }{\kappa +s\left( b\right) },
\]
with ${\bf P}\left(\mathcal  H=\infty \right) =\kappa /\left( \kappa +s\left( \infty
\right) \right) >0$, $s\left( \infty \right) <\infty .$\newline
\end{example}
\begin{remark}
Under the assumptions of Theorem \ref{prop:scale}, let us discuss the situation $s( \infty ) < \infty.  $ In this case we have ${\mathbf P } ( \tau_{x, 0 } < \tau_{x ,\infty } ) < 1 .$ 

Then either $ \tau_{x ,\infty} = \infty  .$ In this case  with 
positive probability the process never comes back to $0$ and thus is transient, that is, converges to $ + \infty $ as $ t \to \infty .$ 

Or $ \tau_{x ,\infty} < \infty , $ such that the process hits state $+ \infty $ even in finite time. Proposition \ref{prop:explos} implies that in this case $ S_\infty < \infty $ such that the jump chain $ Z_n = X_{S_n} $ is transient. 
However in case $ \infty $ is regular, we can add state $+\infty $ to the state space. In this particular situation the process $X_t$ is even recurrent having $ + \infty $ as recurrent state. 
\end{remark}


\subsection{Classification of the recurrence/transience of state $0$ in the separable case}
We close this section with a classification of the recurrence/transience of state $0$ in the separable case with $ h(0) > 0.$ Under Assumption \ref{ass:gammainfty} and \ref{ass:gamma0}, we have :
\begin{itemize}
\item
$ s (\infty ) = \infty , I_0 ( x) < \infty $ : $0$ is recurrent, positive recurrent iff $\int^\infty \frac{h(x)}{\alpha ( x) } e^{ - \Gamma ( x)} dx < \infty .$ 
\item
$ s (\infty ) = \infty , I_0 ( x) = \infty $ : The process is transient in $0$ (almost surely hits $0$ in finite time and stays there forever).
\item
$ s (\infty ) <  \infty , I_{ \infty } (x) = \infty $ : The process is transient (converges to $+ \infty $ with positive probability).
\item
$ s (\infty ) <  \infty , I_{ \infty } (x) < \infty $ : The process is either transient (converges to $+ \infty $ with positive probability) or hits state $ \infty $ in finite time ($\tau_{x, \infty} < \infty $ with positive probability). If state $+ \infty $ is REGULAR, we can add it to the state space, and it will become a recurrent state. If  it is EXIT  the process hits $ + \infty $ in finite time and then stays there forever with positive probability.
\end{itemize}

\subsection{Expected return times to $0$}
This section is devoted to obtain an explicit formula for $ u (x) = \mathbf{E}\left( \tau _{x,0}\right) $ in the case of positive recurrence. In case of total disaster when $H\left( x,0\right) =1$ for all $x,$ we
have $\tau _{x,0}=T_{x}$ which has already been discussed. So we suppose $%
0<H\left( x,0\right) <1$ for all $x$ in this subsection. If $ x > 0, $ we have 
\begin{equation}
\tau _{x,0}\overset{d}{=}T_{x}\mathbf{1}\left( X_{T_{x}}=0\right) +\mathbf{1}%
\left( X_{T_{x}}>0\right) \left( T_{x}+\tau _{X_{T_{x}},0}^{\prime }\right) ,
\label{FSA}
\end{equation}
where $ \tau' _{X_{T_x} } $ is independent of $ {\mathcal F}_{{T_x}}  $ and distributed as  $ \tau_{X_{T_x} }  .$ 
This implies
\begin{equation*}
u(x) =\mathbf{E}\left( \tau _{x,0}\right) =\mathbf{E}\left( T_{x}\right)
+\int_{0^{+}}^{\infty }\mathbf{P}\left( X_{T_{x}}\in dy\right) \mathbf{E}%
\left( \tau _{y,0}\right) , \; x > 0 ,
\end{equation*}
where we recall that 
\begin{equation}\label{eq:u0}
u_{0}\left( x\right) =\mathbf{E}\left( T_{x}\right) =e^{\Gamma \left(
x\right) }\int_{x}^{\infty }\frac{dz}{\alpha \left( z\right) }e^{-\Gamma
\left( z\right) }.
\end{equation}
Imposing Assumption \ref{ass:gammainfty} and \ref{ass:gamma0} and  moreover that  $u_{0}\left( x\right) <\infty ,$
$u_0$ solves
\begin{equation*}
\alpha \left( x\right) u_{0}^{\prime }\left( x\right) -\beta \left( x\right)
u_{0}\left( x\right) =-1, \mbox{ with  }
u_{0}\left( 0\right) =e^{\Gamma ( 0 ) } \int_{0}^{\infty }\frac{dz}{\alpha \left( z\right) }e^{-\Gamma
\left( z\right) },
\end{equation*}
which is finite under Assumption \ref{ass:gamma0},
if $0$ is reflecting. 

In what follows we shall always choose $ \Gamma ( 0 ) = 0 ,$ and we also impose

\begin{ass}\label{ass:more}
1. $X$ is positive recurrent having $0$ as recurrent point. In particular, $H(x, 0 ) > 0 $ for some $ x > 0 $ and $0$ is reflecting, that is $ I_0 ( x) < \infty.$\\
2. The function $ \R_+ \ni x \mapsto \int_0^x  g(y) \bar H(x, dy ) $ is continuous for all bounded measurable test functions $g.$ 
\end{ass}

\begin{proposition}\label{prop:taux}
Suppose that  Assumptions \ref{ass:gammainfty}, \ref{ass:gamma0} and \ref{ass:more} hold.
Suppose moreover that 
$u\left( x\right) =\mathbf{E}\left( \tau _{x,0}\right) $ is locally bounded, that is, $\sup \{  u(y) ,0 \le  y \le x \} < \infty $ for all $x > 0.$ Then $u \in C^1 ( (0, \infty )) ,$ and it solves
\begin{equation}
  {\mathcal G}u  \left( x\right)  =-1 \mbox{  on   } (0 ,\infty) ,   \label{tau1}
\end{equation}
where for all $ x > 0 , $ 
$$  {\mathcal G}u (x)   = \alpha ( x) u'(x) - \beta ( x) H(x, 0 ) u (x) + \beta (x) \int_{0+}^x  \bar H ( x, dy ) [u (y) - u(x) ] . $$
 \end{proposition}

Notice that $ u (0 +) := \lim_{x \to 0} u(x) \neq 0 ,$ implying that in general, $ G u ( x) -  {\mathcal G}u (x)  = \beta ( x) H(x, 0 ) u (0 ) \neq 0. $

\begin{proof}
From (\ref{FSA}), we have
\begin{equation*}
\mathbf{E}\left( \tau _{x,0}\right) =\mathbf{E}\left( T_{x}\right)
+\int_{0^{+}}^{\infty }\mathbf{P}\left( X_{T_{x}}\in dy\right) \mathbf{E}%
\left( \tau _{y,0}\right).
\end{equation*}
If $y>0$\textbf{, }$\mathbf{P}\left( X_{T_{x}}\in dy\right)
=\int_{x}^{\infty }dz\gamma \left( z\right) e^{-\int_{x}^{z}\gamma \left(
z^{\prime }\right) dz^{\prime }}\overline{H}\left( z,dy\right) .$
Therefore

\begin{eqnarray*}
\mathbf{E}\left( \tau _{x,0}\right) &=&\mathbf{E}\left( T_{x}\right)
+\int_{x}^{\infty }dz\gamma \left( z\right) e^{-\int_{x}^{z}\gamma \left(
z^{\prime }\right) dz^{\prime }}\int_{0+}^{z}\overline{H}\left( z,dy\right) 
\mathbf{E}\left( \tau _{y,0}\right) \\
&=&u_0 ( x) +e^{\Gamma \left( x\right)
}\int_{x}^{\infty }dz\gamma \left( z\right) e^{-\Gamma \left( z\right)
}\int_{0+}^{z}\overline{H}\left( z,dy\right) \mathbf{E}\left( \tau
_{y,0}\right),
\end{eqnarray*}
where $u_0$ is given in \eqref{eq:u0} and differentiable on $ (0, \infty ).$ 

Since $ z \mapsto  \gamma \left( z\right) e^{-\Gamma \left( z\right)
}\int_{0}^{z}{H}\left( z,dy\right) \mathbf{E}\left( \tau
_{y,0}\right) $ is continuous, 
$u\left( x\right) =\mathbf{E}\left( \tau _{x,0}\right) $ is differentiable on $ (0, \infty )$ and 
obeys 
\begin{equation*}
u^{\prime }\left( x\right) =u_{0}^{\prime }\left( x\right) +\gamma \left(
x\right) \left( u\left( x\right) -u_{0}\left( x\right) \right) -\gamma
\left( x\right) \int_{0+}^{x} \overline{H}\left( x,dy\right) u\left( y\right).
\end{equation*}
Recalling $u_{0}^{\prime }\left( x\right) =\gamma \left( x\right)
u_{0}\left( x\right) -1/\alpha \left( x\right) ,$ this is  
\begin{equation}\label{eq:12}
u^{\prime }\left( x\right) =-1/\alpha \left( x\right) +\gamma \left(
x\right) \left[ u\left( x\right) -\int_{0+}^{x} \overline{H}\left( x,dy\right)
u\left( y\right) \right] .
\end{equation}
Finally, 
$$
 u\left( x\right) -\int_{0+}^{x} \overline{H}\left( x,dy\right)
u\left( y\right)  = H(x, 0) u(x) + \int_{0+}^x \bar H (x, dy ) ( u(x) - u(y) ),
$$
which implies the assertion.
\end{proof}


%

In what follows, $\pi (y ) $ designs the speed density with integration constant $C $ introduced in \eqref{eq:speed} above. By our assumptions, $\pi (y ) $ is integrable. We also recall the definition of the modified scale function $ s $ in \eqref{eq:s}. We suppose that $H$ is separable with $h(0) > 0.$ In this case it is possible to obtain an explicit formula for $ u (x) $ as we shall show now. 

We start with the following first observation that allows us to determine value of $ u$ in $0,$ $ u(0+ ) := \lim_{x \to 0 } u(x)= \mathbf{E} ( \tau_{0, 0 } ).$ 

\begin{theorem}\label{prop:u(0)}
Grant the assumptions of Proposition \ref{prop:taux}. Let $ \pi $ the unique invariant measure given in \eqref{eq:speed}, where the constant $C$ is chosen such that $ \pi $ is tuned to a probability.  Then for any Borel subset $ B $ of $ \R_+, $ 
\begin{equation}\label{eq:af}
 \pi ( B ) = \frac{1}{u(0+)} \mathbf{E}_0 \int_0^{\tau_{0, 0 } } 1_B ( X_s ) ds .
\end{equation} 
Suppose now moreover that $\pi ( \beta ) \in (0, \infty) $ and that $ \alpha ( 0 ) > 0, $ then
\begin{equation}\label{eq:uin0}
 \mathbf{E} ( \tau_{0, 0 } ) = u(0 +) = \frac{1}{C h(0) }  .
\end{equation} 
\end{theorem} 

\begin{proof}
Representation \eqref{eq:af} is classical and follows from decomposing the trajectory of $ X$ into successive excursions out of $0$ (see e.g. Proposition 2.8 in \cite{DEO}). 
Applying \eqref{eq:af} with $B = [0, \varepsilon] ,$ we obtain 
$$ \frac{1}{ \varepsilon } \int_0^\varepsilon \pi ( y )dy = \frac{1}{u(0+)} \mathbf{E}_0 \left( \frac{1}{ \varepsilon } \int_0^{\tau_{0, 0 }} 1_{ \{X_s \le \varepsilon\} } ds \right) .$$
Letting $ \varepsilon \to 0, $ clearly the left hand side converges to $  \pi (0)  = C (h(0) / \alpha ( 0 )) e^{ - \Gamma ( 0 ) }= C h(0) / \alpha ( 0 ),$ since we have chosen $ \Gamma ( 0 ) = 0.$  The remainder of the proof is devoted to show that 
$$ \lim_{ \varepsilon \to 0} \mathbf{E}_0 \left( \frac{1}{ \varepsilon } \int_0^{\tau_{0, 0 }} 1_{\{ X_s \le \varepsilon\} } ds \right) = 1 / \alpha ( 0 ) .$$ 
Clearly, 
\begin{equation}\label{eq:rem}
  \mathbf{E}_0  \int_0^{\tau_{0, 0 }} 1_{\{ X_s \le \varepsilon\} } ds  = \mathbf {P} ( T_0 > I_0 ( \varepsilon) ) \cdot I_0 ( \varepsilon)  + R ( \varepsilon ) , 
\end{equation}  
with $ I_0 ( \varepsilon ) = \int_0^\varepsilon \frac{1}{ \alpha ( y ) } dy $ the time needed for the deterministic flow to reach $ \varepsilon, $ starting from $0.$ 
In the sequel we will show that the remainder term $R ( \varepsilon ) $ is actually of the order $ R ( \varepsilon ) = o ( \varepsilon).$ Then the assertion follows from 
$$  \lim_{ \varepsilon \to 0}  \frac{1}{ \varepsilon }  \mathbf{P} ( T_0 > I_0 ( \varepsilon) ) \cdot I_0 ( \varepsilon)  =  \lim_{ \varepsilon \to 0}  \frac{1}{ \varepsilon } e^{ - \Gamma ( \varepsilon) } \int_0^\varepsilon \frac{1}{ \alpha ( y ) } dy = \frac{1}{ \alpha ( 0 ) }.$$ 

{\bf Step 1.} In what follows, we shall rely on the fact that for any $ y \geq \varepsilon, $ we have that
\begin{equation}\label{eq:nice}
 \mathbf{E}_y \int_0^{\tau_{y, 0}} 1_{\{ 0 < X_s \le \varepsilon\}} ds = o ( \varepsilon ).
\end{equation}
Indeed, 
\begin{multline}\label{eq:nice2}
 \mathbf{E}_y \int_0^{\tau_{y, 0}} 1_{\{ 0 < X_s \le \varepsilon\}} ds 
 = \sum_{n \geq 1 } \mathbf{E}_y \left( 1_{\{ S_n < \tau_{y, 0 }\}} 1_{\{ X_{S_n} \le  \varepsilon \}}( I_{X_{S_n}} ( \varepsilon ) \wedge (S_{n+1} - S_n) ) \right) \\
 \le I_0 ( \varepsilon ) \sum_{n \geq 1 } \mathbf{E}_y \left( 1_{\{ S_n < \tau_{y, 0 }\}} 1_{\{ X_{S_n} \le  \varepsilon \}}\right)= I_0 ( \varepsilon )  \mathbf{E}_y \left( \sum_{n=1}^{\tau_{y, 0 }- 1} 1_{\{ Z_n \le  \varepsilon \}} \right) .
\end{multline} 
Since $ \pi ( \beta ) < \infty, $ $(Z_n)_n $ is a positive Harris recurrent strong Feller chain (recall Proposition \ref{prop:strongfeller}). Being strong Feller, every bounded measurable function $f$  having compact support is a {\it special function} (see \cite{Revuz}, exercise 4.11, chapter 6, page 215). This means that for any function $h$ such that $ \pi^Z ( h) > 0 , $ the function 
$$ x \mapsto \mathbf{E}_x \left( \sum_{n=1}^\infty ( 1 - h ( Z_1) ) \cdot \ldots \cdot (1 - h ( Z_{n- 1}))  f ( Z_n) \right) $$
is bounded. Taking $ f = 1_{ ( 0 , 1 ]} $ (which, being of compact support, is therefore a special function) and $ h = 1_{ \{ 0\}} $ (which satisfies $ \pi^Z ( h) > 0 $ since $Z_n$ is recurrent coming back to $0$ infinitely often almost surely) we obtain that 
$$ x \mapsto \mathbf{E}_x \sum_{n=1}^{\tau_{x, 0} } 1_{\{Z_n \le 1 \}} \mbox{ is bounded,} $$
implying the assertion by dominated convergence. 

{\bf Step 2.} We now treat the remainder term  $ R (\varepsilon ) = R_1 ( \varepsilon) +R_2 ( \varepsilon) $ appearing in \eqref{eq:rem},  where 
$$ R_1 ( \varepsilon) = \mathbf{E}_0 \left( 1_{\{ T_0 > I_\varepsilon (0 ) \}}  \int_{I_\varepsilon (0) }^{\tau_{0, 0} } 1_{\{ X_s \le \varepsilon\}} ds \right) \mbox{
and }
R_2 ( \varepsilon) = \mathbf{E}_0 \left( 1_{\{ T_0 \le  I_\varepsilon (0 ) \}}  \int_0^{\tau_{0, 0} } 1_{\{ X_s \le \varepsilon\}} ds \right) .$$

Observe that 
\begin{multline*}
R_2 ( \varepsilon ) = \int_0^{I_0 ( \varepsilon )} \beta ( x_t (0) ) e^{ - \int_0^t \beta ( x_s (0) ) ds }dt \\
 \left( t + \int_0^{x_t(0) } \bar H ( x_t (0) , dy ) \mathbf{E}_y \int_0^{\tau_{y, 0 } } 1_{\{ 0 < X_u \le \varepsilon \}} \right)\\
  \le
 \int_0^\varepsilon \gamma ( x) e^{ - (\Gamma ( x) - \Gamma (0 ) )}  \left( I_0 (\varepsilon) + \int_0^{x } \bar H ( x , dy ) \mathbf{E}_y \int_0^{\tau_{y, 0 } } 1_{\{ 0 < X_u \le \varepsilon \}} \right) dx= O ( \varepsilon^2 ) ,
\end{multline*}
since $t \le I_0 ( \varepsilon).$

Concerning the first remainder term, we first use that by 
the Markov property,
$$  R_1 ( \varepsilon) = e^{ - \Gamma ( \varepsilon) } \mathbf{E}_\varepsilon \int_0^{\tau_{\varepsilon, 0 }} 1_{\{ 0 < X_s \le \varepsilon\}} ds .$$ 
We consider three different events. 

We say that event $E_1$  is realized  when the first jump of the process leads to an after-jump position $y \le \varepsilon $ while the second jump of the process happens after the process has reached $\varepsilon.$ 

We say that event $E_2$  is realized  when the first jump of the process leads to an after-jump position $y \le \varepsilon $ and the second jump of the process happens before the process reaches $\varepsilon$ again.
 
We say that event $ E_3$ is realized when  the first jump of the process leads to an after-jump position $y >  \varepsilon . $

Clearly, 
\begin{multline*}
e^{ - \Gamma ( \varepsilon) }  \mathbf{E}_\varepsilon 1_{E_1} \int_0^{\tau_{\varepsilon, 0 }} 1_{\{ 0 < X_s \le \varepsilon\}} ds =
\int_\varepsilon^\infty \gamma ( z) e^{ - \Gamma ( z)  } dz \int_{0+}^\varepsilon \bar H (z, d y ) e^{ - (\Gamma ( \varepsilon) - \Gamma ( y)) } \\
\left(  \int_y^\varepsilon \frac{1}{\alpha (t) } dt + E_\varepsilon \int_0^{\tau_{\varepsilon, 0}} 1_{\{ 0 < X_s \le \varepsilon\}} \right) 
= o ( \varepsilon)  
\end{multline*}
under our hypotheses (where we have used \eqref{eq:nice}).

Similar arguments as those used to control $ R_2 ( \varepsilon ) $ show that 
\begin{multline*}
e^{ - \Gamma ( \varepsilon) }  \mathbf{E}_\varepsilon 1_{E_2} \int_0^{\tau_{\varepsilon, 0 }} 1_{\{ 0 < X_s \le \varepsilon\}} ds \le 
\int_\varepsilon^\infty \gamma ( z) e^{ - \Gamma ( z)  } dz \int_{0+}^\varepsilon \bar H (z, d y ) \int_y^\varepsilon \gamma ( z') e^{ - (\Gamma ( z') - \Gamma (y )) } dz' \\
\left( I_0 (\varepsilon) + \int_{0+}^{z'} \bar H ( z', du ) \mathbf{E}_u \int_0^{\tau_{u, 0}} 1_{\{ 0 < X_s \le \varepsilon \}} ds \right) 
= O ( \varepsilon^2 ).
\end{multline*}
Moreover, using \eqref{eq:nice2},
\begin{multline*}
e^{ - \Gamma ( \varepsilon)}  \mathbf{E}_\varepsilon 1_{E_3} \int_0^{\tau_{\varepsilon, 0 }} 1_{\{ 0 < X_s \le \varepsilon\}} ds \le 
I_0 ( \varepsilon)  \mathbf{E}_\varepsilon \left( 1_{\{ Z_1 \geq \varepsilon \}} \mathbf{E}_{Z_1} ( \sum_{n=1}^{\tau_{Z_1, 0}} 1_{\{ Z_n \le \varepsilon\}} ) \right) \\
\le I_0 ( \varepsilon)  \mathbf{E}_\varepsilon \left(  \sum_{n=1}^{\tau_{Z_1, 0}} 1_{\{ Z_n \le \varepsilon\}}  \right) =  I_0 ( \varepsilon) O ( \varepsilon) .
\end{multline*}
All in all we have shown that $  R ( \varepsilon) = o ( \varepsilon )$ which concludes the proof. 

\end{proof}

\begin{theorem}\label{cor:explicit}
Grant the assumptions of Proposition \ref{prop:taux} together with $ \pi ( \beta ) < \infty, $  and suppose that $ H(x, y ) = h(y)/h(x) , $ where $h$ is differentiable, non-decreasing, with $ h(0) > 0 $ and $\alpha ( 0 ) > 0.$  We choose $ \Gamma ( 0 ) = 0.$ 
%
Then $u ( x) $ is given by 
\begin{eqnarray}\label{eq:19}
 u\left( x\right) &=&u(0) + \int_0^{x}dy\frac{%
\gamma \left( y\right) e^{\Gamma \left( y\right) }}{h\left( y\right) }%
\int_{y}^{\infty }e^{-\Gamma \left( z\right) }\frac{h\left( z\right) }{%
\alpha \left( z\right) }dz
-\int_0^{x}\frac{1}{\alpha \left( y\right) }dy \nonumber \\
&=& u(0) + s( x) \int_x^\infty \pi ( y ) dy + \int_0^x s(y) \pi (y ) dy -\int_0^{x}\frac{1}{\alpha \left( y\right) }dy ,
\end{eqnarray}
where $u(0) $ is given by \eqref{eq:uin0}.
\end{theorem} 

\begin{proof}
We come back to \eqref{eq:12} and 
we put $ \bar h ( y) = h(y ) - h(0)  .$ Using Fubini's theorem and the fact that for $ y > 0, u(x) - u(y) = \int_y^x u' (z) dz, $ since $ u$ differentiable on $ (0, \infty ) ,  $
\begin{eqnarray*} 
u\left( x\right) -\int_{0+}^{x} \overline{H}\left( x,dy\right)
u\left( y\right) 
&=&
 H(x, 0 ) u(x) + \int_{0+}^x (u (x) - u(y) ) \bar H ( x, dy ) \\
&=& H(x, 0 ) u(x) + \int_{0}^x \bar H ( x, y ) u' (y) dy \\
&=&  \frac{h(0) }{h(x)} u(x) + \frac{1}{h(x)} \int_0^x \bar h (  y ) u' (y)dy \\
&=&  \frac{h(0) }{h(x)} u(0) + \frac{1}{h(x)} \int_0^x  h (  y ) u' (y)dy .
\end{eqnarray*}
Therefore, $u$ solves
$$ \alpha ( x) u' ( x) - \frac{\beta (x)}{h(x) } \int_0^x h(y) u' (y) dy    - \frac{\beta (x)}{h(x) }h(0) u(0) = - 1  $$
on $ (0, \infty ) .$

Put $ v (x) = \int_0^x h(y) u' (y) dy  + h(0) u(0)  , $ for $ x > 0 .$  Using integration by parts and the fact that $ h' u \geq 0, $ we obtain that $ v (x) \le h( x) u ( x) < \infty$ for all $x.$ Moreover,  $ v' (x) = h(x) u'(x) $ and $v(0) = h(0) u (0) ,$ and thus 
\begin{equation}\label{eq:vprime}
 v'(x) - \gamma (x) v(x) = - \frac{h(x) }{ \alpha ( x) }  .
\end{equation} 
Putting $ w ( x) := e^{ - \Gamma ( x) } v(x), $ $ w (x) < \infty, $ since $ v(x) < \infty, $ we have 
$$ w'  ( x) = - e^{ - \Gamma( x) } \frac{h(x)}{\alpha ( x) } = -\frac{1}{C} \pi ( x)  < 0 ,$$
where $ \pi $ is the speed density given in \eqref{eq:speed}.

By our assumptions, $ \pi ,$ and hence $ w',$ is integrable on $ \R_+ $ implying that the  explicit solution of the above equation is given by
\begin{equation}\label{eq:w1}
 w ( x) = w(\infty) + \int_{x}^{\infty }e^{-\Gamma
\left( y\right) }\frac{h\left( y\right) }{\alpha \left( y\right) }dy, 
\end{equation}
for some finite constant $ w (\infty ) ,$ 
so that 
\begin{equation}\label{eq:v1}
v\left( x\right) =e^{\Gamma \left( x\right) }\int_{x}^{\infty }e^{-\Gamma
\left( y\right) }\frac{h\left( y\right) }{\alpha \left( y\right) }dy + e^{\Gamma ( x) } w( \infty ) .
\end{equation}
Since by \eqref{eq:vprime}
$$
\frac{v^{\prime }\left( x\right) }{h\left( x\right) } =\gamma \left(
x\right) \frac{v\left( x\right) }{h\left( x\right) }-\frac{1}{\alpha \left(
x\right) }=u^{\prime }\left( x\right) ,$$
this implies
$$u\left( x\right) = u(0) + \int_0^{x}u^{\prime }\left( y\right) dy
=u(0) + \int_0^{x}dy\frac{%
\gamma \left( y\right) e^{\Gamma \left( y\right) }}{h\left( y\right) }%
\left[ \int_{y}^{\infty }e^{-\Gamma \left( z\right) }\frac{h\left( z\right) }{%
\alpha \left( z\right) }dz + w ( \infty ) \right] 
  -\int_0^{x}\frac{1}{\alpha \left( y\right) }dy.
$$
The value of $w ( \infty )  $ is deduced from the fact that on the one hand 
$$ w( 0 ) = w( \infty ) + \int_0^\infty e^{-\Gamma
\left( y\right) }\frac{h\left( y\right) }{\alpha \left( y\right) }dy = w ( \infty ) + \frac{1}{C}$$
and on the other hand
$$ w (0 ) = e^{- \Gamma (0)} v(0) =  h(0 ) u(0) .$$
Replacing $u(0) $ by its explicit value given in \eqref{eq:uin0}, we obtain from this that
$ w (0 ) = \frac{1}{ C}, $
whence $ w( \infty ) = 0, $ which implies the assertion. 
\end{proof}

\begin{example}
Let $h\left( x\right) =e^{x}, $ $ \alpha \left( x\right) =\alpha _{1}x^{a},$\ $a<1$ (entailing $0$
reflecting)$,$ $\beta \left( x\right) =\beta _{1}x^{a}$, ($b=a>a-1$).
Assumptions $1$ and $2$ are satisfied. To ensure recurrence, we assume $%
\gamma \left( x\right) =\gamma _{1}>1$ and due to this, we obtain the
expected first return time to $0$ as 
\[
u\left( 0\right) ={\bf E}\left( \tau _{0,0}\right) =\frac{1}{\alpha _{1}}%
\int_{0}^{\infty }y^{-a}e^{-\left( \gamma _{1}-1\right) y}=\frac{\Gamma
\left( 1-a\right) }{\alpha _{1}\left( \gamma _{1}-1\right) ^{1-a}}<\infty .
\]
Note that, consistently, $u\left( 0\right) $ diverges when $\gamma
_{1}\downarrow 1$\ and also when $a\uparrow 1.$ We also have
\begin{eqnarray*}
u\left( x\right)  &=&u\left( 0\right) +\frac{\gamma _{1}}{\left( \gamma
_{1}-1\right) \alpha _{1}}\int_{0}^{x}de^{\left( \gamma _{1}-1\right)
y}\int_{y}^{\infty }\frac{e^{-\left( \gamma _{1}-1\right) z}}{z^{a}}dz-\frac{%
1}{\alpha _{1}\left( 1-a\right) }x^{1-a} \\
&{\sim }&\frac{1}{\left( \gamma
_{1}-1\right) \alpha _{1}\left( 1-a\right) }x^{1-a} \; \mbox{ as } x \to \infty,
\end{eqnarray*}
where, after integration by parts, we used a large $x$ estimate of the
integral Gamma function. The large $x$ expected time to local extinction is
algebraic. An exact expression (involving the integral Gamma function) of $%
u\left( x\right) $ for all $x$ is available from the first expression of $%
u\left( x\right) $.$\newline
$
\end{example}

\section{Some Simulations}
We illustrate our results by some simulations involving a growth model with immigration. In our simulations we take $\alpha(x)=\alpha_0+\alpha_1 x^a$ and $\beta(x)=x^b $ with $\alpha_0=\alpha_1=1, $ $a=2$ and $b=\frac{3}{2}.$
In this case, the state $0$ is reflecting, and the process $x_t(x) $ reaches $ \infty $ in finite time. Assumptions \ref{ass:gammainfty} and \ref{ass:gamma0} are both satisfied.
We work in the separable case $H(x,y)=\frac{h(y)}{h(x)} .$

The following simulations are done in discrete time  by using  the embedded chain $ Z_n = X_{S_n}$ in the case where $ 0 $ is not absorbing. In this case,  we have for all  $x\geq 0, $ 
\begin{eqnarray*}
{\bf P}\left( Z_{n}\in dy\mid Z_{n-1}=x\right) &=&e^{\Gamma \left( x\right) }\int_{x}^{\infty }dz\gamma \left( z\right)
e^{-\Gamma \left( z\right) }H\left( z,dy\right) ,
\end{eqnarray*}
translating that $Z_{n}$ is a time-homogeneous discrete-time Markov chain on 
$\left[ 0,\infty \right] $.

We also have 
\begin{multline}\label{S1}
{\bf P}\left( Z_{n}\leq y\mid Z_{n-1}=x\right) =e^{\Gamma \left( x\right)
}\int_{x}^{\infty }dz\gamma \left( z\right) e^{-\Gamma \left( z\right)
}\int_{0}^{y}H\left( z,dy^{\prime }\right)
\\
=1-e^{-\left( \Gamma \left( x\vee y\right) -\Gamma \left( x\right) \right)
}+\int_{x\vee y}^{\infty }dz\gamma \left( z\right) e^{-\left( \Gamma \left(
z\right) -\Gamma \left( x\right) \right) }H\left( z,y\right)  .
\end{multline}
Indeed, since $H\left( z,y\right) =1$\ for all $y\geq z$\ and only whenever 
$y\geq x$, the second integral in the first equation has to be cut into two
pieces corresponding to ($z>y$ and $x<z\leq y$). 


To simulate the embedded chain, we have to decide first if, given $Z_{n-1}=x,$ the forthcoming move is down or up.

- A move down occurs with probability 
$
{\bf P}\left( Z_{n}\leq x\mid Z_{n-1}=x\right) =\int_{x}^{\infty }dz\gamma
\left( z\right) e^{-\left( \Gamma \left( z\right) -\Gamma \left( x\right)
\right) }H\left( z,x\right) { .}
$

- A move up occurs with complementary probability.

As soon as the type of move is fixed (down or up), to decide where the
process goes precisely, we must use the inverse of the corresponding
distribution function (\ref{S1}) (with $y\leq x$\ or $y>x$), conditioned on
the type of move.

\begin{remark}
$\left( i\right) $ If the jump kernel $H\left( z,y\right) $\ is decreasing
in $z$\ for each fixed $y$, then, from (\ref{S1}), the embedded chain is
stochastically monotone, that is, for each fixed $y$, ${\bf P}\left(
Z_{n}\leq y\mid Z_{n-1}=x\right) $\ is decreasing in $x$. Note that 
$$
{\bf P}\left( Z_{n}\in dy\mid Z_{n-1}=x\right)  =e^{\Gamma \left( x\right)
}\int_{x}^{\infty }dz\gamma \left( z\right) e^{-\Gamma \left( z\right)
}H\left( z,dy\right)  
={\bf E}H\left( G\left( x\right) ,dy\right) .
$$

$\left( ii\right) $ If state $0$\ is absorbing, equation (\ref{S1}) is valid only
when $x>0$\ and the boundary condition ${\bf P}\left( Z_{n}=0\mid
Z_{n-1}=0\right) =1$\ should be added. 
\end{remark}

The first simulation is done with the choice $h(x)=e^x.$ Here,  state $ +\infty$ is an absorbing state.
\begin{center}
\includegraphics[scale=0.3]{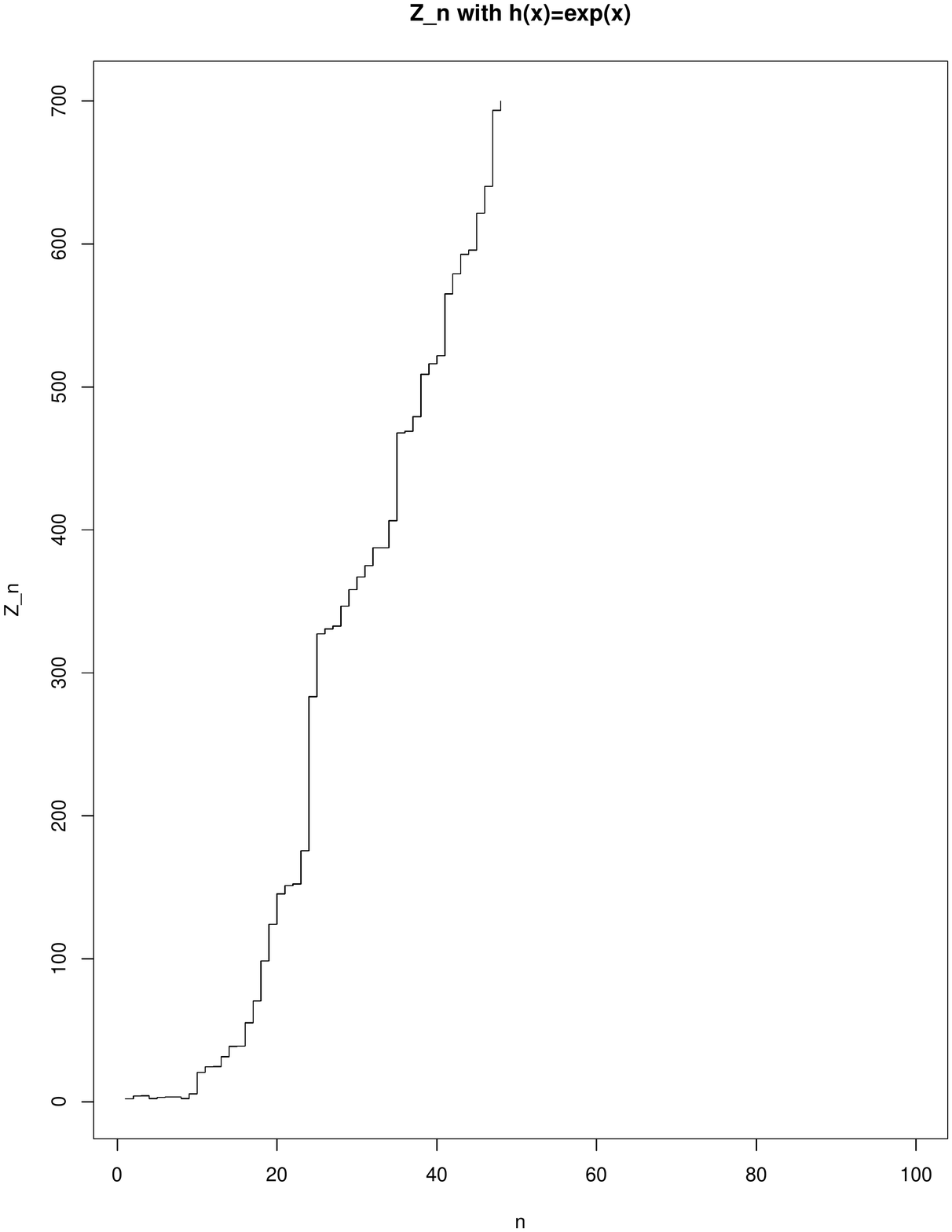}
\end{center}

We can remark the occurrence of many jumps for small values of the process and the scarcity of jumps for large values. In other words, the probability of disaster when the process is at position $x$ tends to $0$ when $x$ tends to infinity. It is decreasing in $x$, i.e. the greater $x$ is, the less is the probability of disaster at that point. In particular, $ H ( \infty , \{ \infty \} ) = 1, $ that is, state $ + \infty $ is absorbing.

By doing a simple calculation we notice that $s(\infty) < \infty $ and $I_{\infty}(x) < \infty .$ Using the last criterion in section $4.4 $ we conclude that either the process $X$ is transient (converges to $ +\infty $ as $t \to \infty $) or hits  $+\infty $ in finite time and then stays there forever.

In the next simulation we choose $h(x)=1 $  for all $ x $  (total disaster case). 
In this case $ Z_n = 0 $ for all $ n \geq 1.$ To obtain some information about the process, in this case we have simulated $ U_n = X_{S_n-} .$ Since $ s( \infty ) = \infty,  $ the process $X$ is recurrent and comes back to $ 0$ infinitely often. 
We have 
\begin{eqnarray*}
{\bf P}\left( U_{n}\in dy\mid U_{n-1}=x\right)  &=&\int_{0}^{x } H ( x, dz) \int_0^\infty dt\beta
\left( x_{t}\left( z\right) \right) e^{-\int_{z}^{x_{t}\left( z\right)
}\gamma \left( u\right) du} \delta_{ x_t ( z) } (dy )   \\
&=&\int_0^x H(x, dz) e^{\Gamma \left( z\right) }\int_{z}^{\infty }du\gamma \left( u\right)
e^{-\Gamma \left( u\right) } \delta_u( dx) .
\end{eqnarray*}
In the particular case $ h ( x) = 1 , $ that is, $ H(x, dz ) = \delta_0 ( dz) , $ this gives
\begin{equation}\label{eq:u}
 {\bf P}\left( U_{n}\in dy\mid U_{n-1}=x\right) = \gamma ( y ) e^{ - ( \Gamma ( y ) - \Gamma ( 0 ) )} dy,
\end{equation} 
that is, $ (U_n)_{n \geq 1 } $ is an i.i.d. sequence with common distribution given according to \eqref{eq:u}.
 
\begin{center}
\includegraphics[scale=0.3]{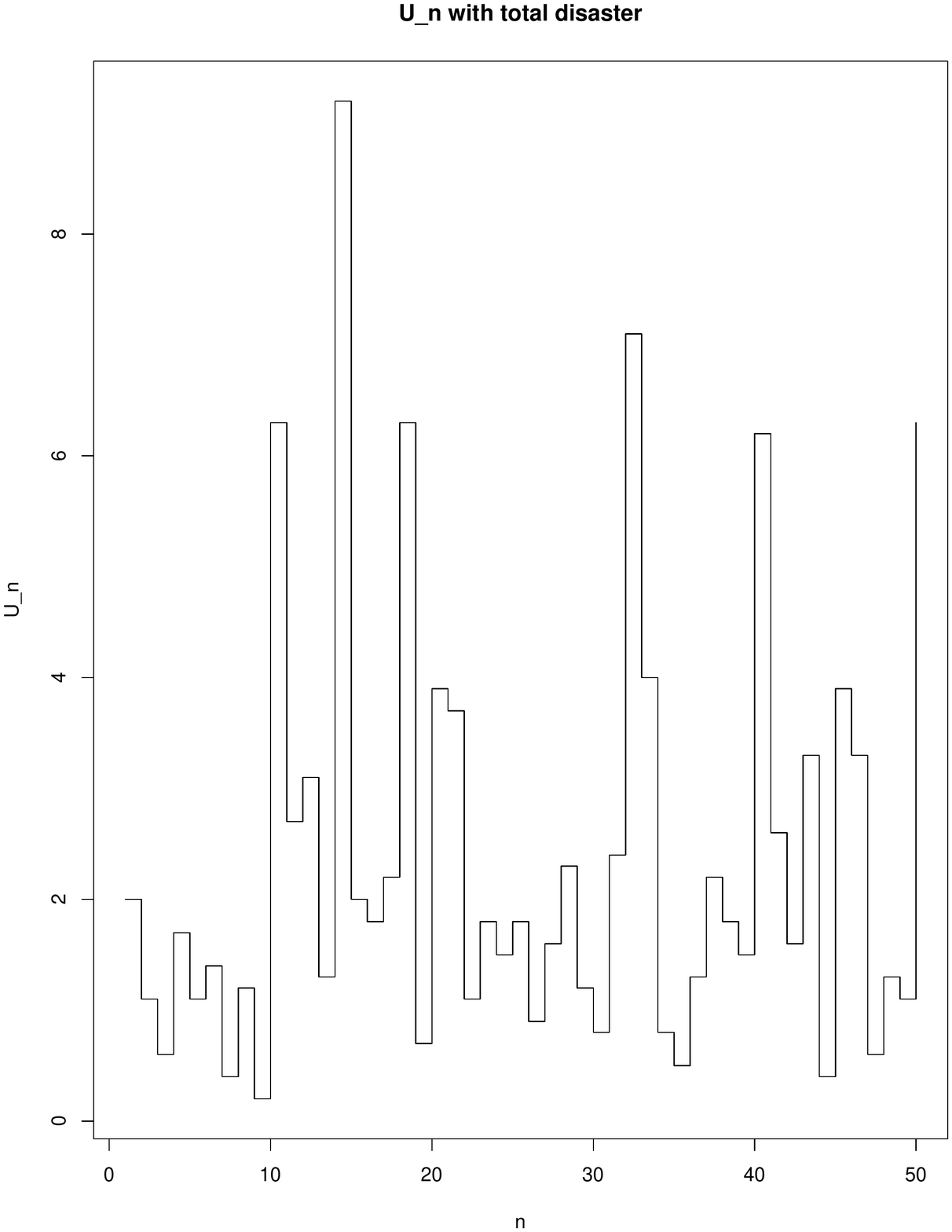}
\end{center}

\textbf{Acknowledgments:}\\
T. Huillet acknowledges partial support from the ``Chaire \textit{%
Mod\'{e}lisation math\'{e}matique et biodiversit\'{e}''.} B. Goncalves and T. Huillet 
acknowledge support from the labex MME-DII Center of Excellence (\textit{%
Mod\`{e}les math\'{e}matiques et \'{e}conomiques de la dynamique, de
l'incertitude et des interactions}, ANR-11-LABX-0023-01 project). Finally, this work was
also funded by CY Initiative of Excellence (grant ``Investissements
d'Avenir'' ANR- 16-IDEX-0008), Project EcoDep PSI-AAP 2020-0000000013$.$

\end{document}